\documentclass[11pt,reqno,tbtags]{amsart}

\usepackage{amsthm,amssymb,amsfonts,amsmath}
\usepackage{amscd} 
\usepackage[numbers, sort&compress]{natbib} 
\usepackage{subfigure, graphicx}
\usepackage{vmargin}
\usepackage{setspace}
\usepackage{paralist}
\usepackage[pagebackref=true]{hyperref}
\usepackage[usenames,dvipsnames]{color}
\usepackage[normalem]{ulem}
\usepackage{stmaryrd} 
\usepackage[refpage,noprefix,intoc]{nomencl} 
\usepackage{cleveref}

\providecommand{\N}{}

\renewcommand{\N}{{\mathbb N}}

\newcommand{\E}[1]{{\mathbf E}\left[#1\right]}				
\newcommand{\e}{{\mathbf E}}

\newcommand{\p}[1]{{\mathbf P}\left\{#1\right\}}

\newcommand{\I}[1]{{\mathbf 1}_{[#1]}}

\newcommand\cA{\mathcal A}
\newcommand\cB{\mathcal B}
\newcommand\cC{\mathcal C}
\newcommand\cD{\mathcal D}

\newcommand\cF{\mathcal F}
\newcommand\cG{\mathcal G}
\newcommand\cH{\mathcal H}

\newcommand\cJ{\mathcal J}

\newcommand\cM{\mathcal M}

\newcommand\cS{{\mathcal S}}

\newcommand\cX{{\mathcal X}}

\providecommand{\eps}{}
\renewcommand{\eps}{\epsilon}
\providecommand{\ora}[1]{}
\renewcommand{\ora}[1]{\overrightarrow{#1}}
\DeclareRobustCommand{\SkipTocEntry}[5]{} 

\reversemarginpar
\definecolor{clou}{rgb}{0.8,0.25,0.5125}

\newtheorem{thm}{Theorem}
\newtheorem{lem}[thm]{Lemma}
\newtheorem{prop}[thm]{Proposition}
\newtheorem{obs}[thm]{Observation}
\newtheorem{cor}[thm]{Corollary}

\newtheorem{conj}{Conjecture}

\newtheorem{claim}[thm]{Claim}

\newenvironment{claimproof}{\begin{proof}[Proof of Claim]}{\end{proof}}

\numberwithin{equation}{section}
\numberwithin{thm}{section}

\makenomenclature										
\graphicspath{ {./images/} }
\usepackage{wrapfig}

\newcommand{\cal}{\mathcal}
\newcommand{\ex}{\mathrm{ex}}
\newcommand{\out}{\mathrm{out}}
\newcommand{\cond}{\mathrm{cond}}
\newcommand{\D}{\mathcal D}
\newcommand{\mix}{\mathrm{mix}}
\newcommand{\diam}{\mathrm{diam}}
\newcommand{\cyc}{\mathrm{cyc}}
\newcommand{\Cyc}{\mathrm{Cyc}}
\newcommand{\ones}{a}

\usepackage{cleveref}

\usepackage{thmtools}
\newtheorem{innercustomthm}{Theorem}
\newenvironment{customthm}[1]
  {\renewcommand\theinnercustomthm{#1}\innercustomthm}
  {\endinnercustomthm}

\begin{document}

\title[Mixing for giant component of random graphs with given degrees]{Diameters and mixing times for giant components of random graphs with given degrees} 
\author{Louigi Addario-Berry}
\address{Department of Mathematics and Statistics, McGill University, Montr\'eal, Canada}
\email{louigi.addario@mcgill.ca}

\author{Bruce Reed}
\address{Institute of  Mathematics, Academia Sinica, Taipei, Taiwan}
\email{bruce.al.reed@gmail.com}

\author{Corrine Yap}
\address{School of Mathematics, Georgia Institute of Technology, Atlanta, GA, USA}
\email{math@corrineyap.com}

\date{May 15, 2026} 


\begin{abstract} 
A sequence ${\mathcal D}=\{d_1,...d_n\}$ is a {\it feasible degree sequence} 
if there is a  graph on $\{1,...,n\}$ such that $i$ has degree $d_i$.
For such a sequence, $G({\mathcal D})$ is a graph chosen uniformly at random from those with the given degree sequence. 
We consider sequences $\{\cD_\ell\}_{\ell \ge 1}$ of feasible degree sequences which have a giant component. We show that with high probability this giant component is  unique, and bound its diameter and the mixing time of the random walk on it. We also bound the size and diameter of the other components and show that many of these bounds are tight.
\end{abstract}

\maketitle



\tableofcontents

\section{Introduction}\label{sec:intro}

A sequence ${\mathcal D}=(d_1,\ldots,d_n)$ where $d_1 \leq d_2 \leq \cdots \leq d_n$ is a {\it feasible degree sequence} 
if there is a 
(simple) graph on $[n]:= \{1,...,n\}$ such that $i$ has degree $d_i$.
Given a feasible degree sequence ${\mathcal D}$, define $\cG(\cD)$ to be the set of all (simple) graphs on $[n]$ with degree sequence $\cD$. Let $G=G({\mathcal D})$ be a uniformly random element of $\cG(\cD)$, and let $m=m(G) =m(\cD)=\frac{\sum_{i=1}^n d_i}{2}$ be the number of edges of $G$. 

Fix a sequence $\{\cD_{\ell}\}_{\ell \geq 1}$ of feasible degree 
sequences, where ${\cal D}_{\ell}=(d_{\ell,1},\ldots,d_{\ell,n(\ell))})$ and $n(\ell) \to \infty$ as $\ell \to \infty$. We say $\{\cD_{\ell}\}_{\ell\geq 1}$ {\em has a giant component} if there exists $\epsilon>0$ such that the probability that $G({\cal D}_{\ell})$ has a component containing at least $\epsilon n(\ell)$ vertices goes to 1 as~$\ell$ goes to infinity. We will sometimes say ``$G(\cD_{\ell})$ has a giant component'' to mean that $\{\cD_{\ell}\}_{\ell\geq 1}$ has a giant component. 

We shall use the phrase ``with high probability'' or ``whp'' to mean ``with probability tending to $1$ as some other parameter goes to infinity.''
If the parameter is unspecified it is the index $\ell$ of the degree 
sequence in a sequence of degree sequences.

In this work, we show that if $G(\cD_{\ell})$ has a giant component, then with high probability it has only one component of order much larger than $n(\ell)^{1/3}$ and bound the diameter and the asymptotics of the mixing time of the lazy random walk on all of  its components.

Before describing this walk and discussing our results, we discuss some properties that a sequence of degree sequences which has a giant component 
must satisfy. For a degree sequence ${\cal D}=(d_1,\ldots,d_n)$, and integer $d \ge 0$, it is useful to write $n_d=n_d(\cD)=|\{i \in [n]: d_{i}=d\}|$ for the number of vertices with degree $d$ in $G(\cD)$, and also to write $n(\cD)=n$.

We first consider vertices of degree zero. If there is no $\alpha>0$ such that the proportion  of vertices of degree 
zero is at most  $1-\alpha $ for large enough $\ell$, then $G(\cD_{\ell})$ has no giant component. Otherwise, $\{\cD_{\ell}\}_{\ell \ge 1}$ has a giant component precisely if the sequence of degree sequences obtained by deleting the  zeroes in every degree sequence does. So we assume $n_0=0$ for all degree sequences we consider.

We next consider the vertices of degree 2. For any degree sequence ${\cal D}$, the graph
$G(\cD)$ consists of the union of some cycle components and some components which are obtained by replacing with paths the edges of some multigraph $G^{\neq 2}$, all of whose vertices have degree which is not two. The multigraph $G^{\ne 2}$ has $m^{\ne 2}=m^{\neq 2}(\cD):=m-n_2(\cD)$ edges. Moreover, conditioned on the set $S$ of vertices in cycle components of $G(\cD)$, the subgraph $G(\cD)|_S$ induced by $S$ is a uniformly random $2$-regular graph with vertex set $S$.

If the union of the cycle components of a graph $G$ with $n$ vertices contains more than  $(1-\epsilon)n$ vertices then $G$ has a  component of order $\epsilon n$ if and only if this union does.  Using this, it was shown by Joos {\em et al} \cite[Theorem 6]{joos2018determine} that if, for some constant $B$, there are infinitely many $\ell$ for which $m^{\neq 2}(\cD_{\ell}) \le B$, then  $\{\cD_{\ell}\}_{\ell\geq1}$ has no giant component. So, in order for there to be a giant component, we must have $m^{\neq 2}(\cD_{\ell})=\omega(1)$, in which case, as we shall show below, with probability $1-o(1)$ the cycle components  contain $o(n(\cD_{\ell}))$ vertices.

For a graph $H=(V,E)$, the {\em lazy random walk on $H$}, when at vertex $u$, stays at $u$ with probability $\frac{1}{2}$ and moves to each neighbour of $u$ with probability $\frac{1}{2\deg(u)}$, where $\deg(u)=\deg_H(u)$ is the degree of $u$ in $H$. If $H$ is connected,
there is a unique stationary distribution $\pi$, meaning if we choose a vertex according to $\pi$ and perform a step of the walk, the resulting distribution is still $\pi$. 
As can be easily verified, this distribution satisfies 
$\pi(u) =\frac{\deg(u)}{2|E|}$. 

Let $P$ be the transition matrix of the lazy walk, so $P_{i,j}$ is the probability we move to $i$ if we are at $j$. 
If we start our walk at time $t=0$ from vertex $i$, then our starting probability distribution 
is $\mu^{0,i}=(\mu^{0,i}_v)_{v \in V}$ with $\mu^{0,i}_i=1$ and $\mu^{0,i}_v=0$ for $v \ne i$. For $t \ge 1$,  $\mu^{t,i}=\mu^{0,i} P^t$  is the probability distribution of the endpoint of a $t$-step lazy walk that started at $i$. In a finite connected graph, for every $i$, $\mu^{t,i}_i$ converges to $\pi$ 
in the following sense. The {\em total variation distance} $d_T(P,P')$ between two probability distributions
$P$ and $P'$ on a state space $\Omega$ is defined as $\max_{A \subseteq \Omega}|P(A)-P'(A)|$. We have that for all $i$,
the total variation distance of $\mu^{t,i}$ and $\pi$ tends to zero as $t$ goes to infinity. 
This is not true for a non-lazy walk which simply chooses a uniform random neighbour; for instance, if the graph is regular and bipartite, then $\pi$ is the uniform distribution 
and hence for each side $A$ of the bipartition and every time $t$ and starting vertex $i$,
$|\mu^{t,i}(A)-\pi(A)|\ge\frac{1}{2}$.  It is for this reason that we focus on the lazy random walk.

The {\em mixing time} of the lazy random walk on $H=(V,E)$, which measures how quickly the walk converges to $\pi$ from the 
worst starting point, is defined as 
\[
\tau_{\mix}=\tau_{\mix}(H) := \sup_{i\in V}\min\left\{t\ \left| \ 
d_T(P^t_i,\pi) < \frac{1}{e}\right.\right\}.
\]
From here on, we write ``the mixing time of $H$'' to mean $\tau_{\mix}(H)$. The graph on $n$ vertices with the smallest mixing time is a clique, which has mixing time~$2$ for $n \ge 3$. 
One connected graph with asymptotically the largest mixing time is the ``barbell graph'' 
obtained from the union of two cliques with $\frac{n}{4}$
vertices each 
and a path with $\frac{n}{2}$ vertices by adding an edge from one vertex of each clique to distinct 
endpoints of the path   (call these endpoints the start and the end of the path). For this graph, if the random walk starts in one clique, then until the first time it has probability at least $\frac18$ of having reached the other clique, the total variation distance is at least $\frac12-\frac18-o(1) >\frac1e$ for $n$ large, so the random walk is not yet mixed. Standard random walk theory tells us that if we begin a random walk at the start of the path, the probability that it exits  the path from the end ---into the other clique---is $\frac{2}{n+1}$. On the other hand, if we exit at the start of the path, then it takes on average $\frac{n}{4}$ visits to the unique clique vertex incident to the start of the path before we revisit the start of the path, and the average time between such visits is also $\frac{n}{4}$; so with high probability it takes $\Omega(n^2)$ time to return to the start of the path. It follows that, starting from one clique, with high probability it takes time $\Omega(n^3)$ to reach the other clique, and so the mixing time is $\Omega(n^3)$.

The distance between two vertices $x, y$ of a graph, denoted $\mathrm{dist}(x,y)$, is the number of edges on a shortest path between them. The {\em diameter} of a graph $H$, denoted $diam(H)$, is the maximum  of the distance  between two of its vertices.
It is not hard to see that the mixing time is at least half the diameter, and it is often large diameter graphs which have high mixing times, as in the example just given. However, graphs with small diameters can also have large mixing time. For example, a  graph with diameter three  and high mixing time is obtained from  the disjoint union of two 
cliques of order $n/2$ by adding an edge between them. Its mixing time  is essentially the time until with probability 
at least $\frac{1}{2}-\frac{1}{e}$ we move from one clique to the other which is   $\Omega(n^2)$. 

If ${\cal D}=(d_1,\ldots,d_n)$ is the constant sequence $(n-1, \cdots, n-1)$ then 
$G({\cal D})$ is a clique, so we can achieve the 
quickest mixing time with a random graph on a fixed degree sequence. 
However, our results imply  that typical graphs with a given degree sequence $(d_1,\ldots,d_n)$ cannot have mixing time as large as $n^3$.
In the coming theorem statements, we write $n=n(\cD_{\ell})$, $m=m(\cD_{\ell})$, and $m^{\ne 2}=m^{\ne 2}(\cD_{\ell})$. 

\begin{thm}\label{thm:mainnew1} For any $h$ going to infinity with $n$, the following holds. Let $(\cD_{\ell})_{\ell \ge 1}$ be a sequence of feasible degree sequences such that  $n_0(\cD_{\ell})=0$ and $G(\cD_{\ell})$  has a giant component. Then whp  
    the diameter of every  component  is $O(\frac{m \log m^{\neq 2}}{m^{\neq 2}})$  and all  but the largest contain at most one cycle and  have $O(\max\{n^{1/3}h(n), \frac{m \log m^{\neq 2}}{m^{\neq 2}}\})$ vertices.  
\end{thm}

\begin{thm}\label{thm:mainnew2} For any $h$ going to infinity with $m$, the following holds. Let $(\cD_{\ell})_{\ell \ge 1}$ be a sequence of feasible degree sequences such that $n_0(\cD_{\ell})=0$,  $G(\cD_{\ell})$  has a giant component, and  $m \ge h(m)m^{\neq 2}$. Then  whp the mixing time  of every component is  $O((\frac{m\log m^{\neq 2}}{m^{\neq 2}})^2)$. 
\end{thm}

We remark that if $m>100 m^{\neq 2}$, then $m<3n$ which implies 
$\frac{m \log m^{\neq 2}}{m^{\neq 2}} =o(n)$. 

To bound the mixing time without assumptions on $\frac{m}{m^{\neq 2}}$, we use the following theorem, which allows us to combine our previous two results.

\begin{thm}[{\cite[Corollary 4.2]{peres}}]
\label{diameterboundonmix}    
    The mixing time of any connected graph $H$ is  at most $8\,\diam(H)|E(H)|$.
\end{thm}

As a result, we obtain the following.

\begin{cor}\label{cor:mainnew3} For any $h$ going to infinity with $m$, the following holds. Let $(\cD_{\ell})_{\ell \ge 1}$ be a sequence of feasible degree sequences. 
    If  $n_0(\cD_{\ell})=0$ and $G(\cD_{\ell})$  has a giant component, then  whp 
    the mixing time of every  component is  $O( n^2 h(m) \log n  )$. 
\end{cor}

\begin{proof}
    We partition our sequence of degree sequences  up into two sequences, 
    the first consisting of those degree sequences for which $m \ge h(m) m^{\neq 2}$ 
    and the second consisting of those for which $m <h(m)  m^{\neq 2}$. 
    We apply Theorem \ref{thm:mainnew2} to the  first sequence, and Theorems \ref{thm:mainnew1}
    and \ref{diameterboundonmix} to the second, and use that $m^{\ne 2}=O(n^2)$. 
\end{proof}

We shall show below that the  bound on the order of the second largest component   in Theorem \ref{thm:mainnew1}  is   tight  up to the $h(n)$ factor. 
  We claim the diameter bound in Theorem \ref{thm:mainnew1} is tight up to a $O(\log m^{\neq 2})$,  and the mixing time bound in Theorem~\ref{thm:mainnew2}  is tight up to a $O((\log m^{\neq 2})^2)$ factor. To prove our claim, for  any function $f(n)$ which is $o(n)$ but  goes to infinity as $n \to \infty$,  consider the degree sequence of length $n$ which has $2 \lceil f(n) \rceil$
 elements which are $\lceil f(n) \rceil$ and the rest of which are two. Any realization of this sequence 
 consists of cycle components and components containing at least one  vertex of degree $\lceil f(n)\rceil$. 
 Lemma 35 of \cite{joos2018determine}\footnote{The $M$ of that lemma is $2m^{\neq 2} = 4f(n)^2$, and its quantity $n'_2$ is at least $n_2-m^{\neq 2}=1-o(1)n$.} implies that  with high probability the cycle components have total size $o(n)$, and Lemma 28 of \cite{joos2018determine}\footnote{Again $M=2m^{\neq 2}=4f(n)^2$, and $L$ is the set of vertices of degree at least  
 $\frac{\sqrt{M}}{\log M}$ which is precisely the set of the $2f(n)$ high-degree vertices.} implies  whp  all of the high-degree 
 vertices lie in the same component. The degree-two vertices in this  giant component are partitioned into 
 paths which have endpoints joined to high-degree vertices. One of these  paths $P$  has length 
 at least $\frac{n}{f(n)^2}$. This implies that the component's diameter is at least $\frac{n}{f(n)^2}$ and that its mixing time is at least the minimum $t$ such that if we start a lazy random walk at the midpoint of $P$, then the probability it has moved halfway to one of the endpoints within the first $t$ steps is at least $\frac{1}{2}-\frac{1}{e}$; this quantity is $\Omega(\frac{n^2}{f(n)^4})$. Since $m^{\neq 2} =\theta(f(n)^2)$ this proves our claim. 

 We can improve on our bounds  provided  the maximum degree $\Delta(G)=d_n$ does not grow too quickly, obtaining the following.

 \begin{thm}\label{thm:mainnew4} For any $h$ going to infinity with $m$, the following holds. Let $(\cD_{\ell})_{\ell \ge 1}$ be a sequence of feasible degree sequences. 
    If  $n_0(\cD_{\ell})=0$, $G(\cD_{\ell})$  has a giant component, and $\Delta=o(\sqrt{m})$  then whp all components but the largest have 
    $O(h(m)\frac{m \log m^{\neq 2}}{m^{\neq 2}})$ vertices 
    and every component has diameter
     $O(h(m)\frac{m \log m^{\neq 2}}{m^{\neq 2}})$ and mixing time $O(h(m)(\frac{m \log m^{\neq 2}}{m^{\neq 2}})^2)$.
\end{thm}

The bound of this theorem is tight,  as witnessed by an example we now sketch. The {\em core} of a graph $G$ is the maximal subgraph of $G$ of minimum degree 2. $G$ is obtained from its core by adding {\it decorations}, which are trees that are disjoint from each other and from the core, each of which is either a component of $G$ or is attached to  a single vertex of the core by a single edge.

Fix $n=3k$ and consider the degree sequence 
\[
\cD=(\,\underbrace{3, \ldots, 3}_{\text{$2k$ times}}\;\;,
  \underbrace{1, \ldots, 1}_{\text{$k$ times}}\; ).
\]
It is straightforward to verify that with high probability $G=G(\cD)$ has a giant component of order $\Theta(n)$, and that there are decorations of order $\Theta(\log n)$ and height $\Theta(\log n)$ attached to the core of the giant component. If the random walk begins within such a decoration, at a vertex of maximal distance from the core, then it takes time $\Theta(\log^2 n)$ to reach the core, and until it does so it has not mixed; therefore, the mixing time is $\Omega(\log^2 n)$.

To see that a bound on $d_n$ is necessary in this theorem, consider the sequence of degree sequences 
  with $d_n=d_{n-1}=\frac{n}{2}=\frac{m+1}{2}$, and $d_i=1$ for $i \in [n-2]$. For such sequences, $G({\mathcal D})$ always consists of two stars, each on $\frac{m+1}{2}$ vertices, with an edge between their centers. In this case, the mixing time is bounded from below by the minimum $t$ for which the probability we have moved from one star to the other after $t$ steps is at least $\frac{1}{2}-\frac1e$, which is $\Omega(n)$. 

Similarly, given $D=o(\sqrt{n})$ with $\frac{n}{D}$ an integer, we may consider a degree sequence $\cD$ with $d_i=\Delta=\frac{n}{D}+D-2$ for  $n-D < i \leq n$ and $d_i=1$ for all other $i$. For such degree sequences, $G(\cD)$ always consists of a clique of $D$ vertices, to each of which $\frac{n-D}{D}$ leaves is attached. On this graph, mixing requires having a decent probability of visiting at least two vertices of the clique, which takes time $\Omega(\frac{n-D}{D^2})$. Since $\Delta=\frac{n}{D}+D-2$ and $D=o(\sqrt n)$, implying $m=(1+o(1))n$, this mixing time is $\Omega(\frac{\Delta^2}{m})$.  Thus, the strengthening of Theorem \ref{thm:mainnew4} obtained by replacing $o(\sqrt{m})$ by $o( \sqrt{m} (\log m)^{3/2})$ is false. However, we believe that the following is true.
\begin{conj}
    \label{con:mainnew5} For any $h$ going to infinity with $m$, the following holds. Let $(\cD_{\ell})_{\ell \ge 1}$ be a sequence of feasible degree sequences. 
    If  $n_0(\cD_{\ell})=0$ and  $G(\cD_{\ell})$  has a giant component,  then whp the mixing time of every  component  is  $O(h(m)(\log m^{\neq 2})^2 \max( \frac{\Delta^2}{m}, (\frac{m}{m^{\neq 2}})^2))$.  
\end{conj}

\subsection{The existence of giant components and the configuration model} \label{sec:configmodel} $~~$\\
We now present the conditions which ensure the existence of a giant component
when $n_0=0$ and $m^{\neq 2}=\omega(1)$. 

Fix a degree sequence $\cD=(d_1,\ldots,d_n)$. 
We can generate $G(\cD)$ using the configuration model (introduced in \cite{BBK1972}) as follows. For $1 \le i \le n$, we take a tuple of $d_i$ labelled halfedges $i1,\ldots,id_i$ corresponding to $i$ 
and then choose a random matching on these halfedges. This may not generate a simple graph, as there can be loops 
and multiple edges. However, if we condition on  choosing a random matching  such that the multigraph is simple, we recover a graph with the distribution of $G(\cD)$.
We let $M=M(\cD )$ be this conditioned random matching. In other words, writing $\cM=\cM(\cD)$ for the set of perfect matchings of $\bigcup_{i \in [n]}\{i1,\ldots,id_i\}$ that yield a simple graph, then $M \in_u \cM(\cD)$ is a uniformly random element of $\cM(\cD)$. 

In both the configuration model, where our choice of a matching is unconditioned, and the uniform simple graph model,  where it is not, we can think of exploring the component containing a given vertex $i$ by growing a spanning tree for it.
We start with a tree consisting of vertex $i$, together with $d_i$ unexplored open halfedges leaving the tree. 
In each iteration we expose the halfedge matched to one of these open halfedges. 
If the matching halfedge is already in the tree, we lose two open  halfedges. Otherwise, if the matching halfedge 
corresponds to some vertex $j$ not in the tree, we lose one open halfedge but gain $d_j-1$ open halfedges from $j$ 
(one of the halfedges from $j$ is used to link into the existing tree). 

In the configuration model, the probability that in the first iteration we link to a halfedge corresponding to some $j \neq i$ is $\frac{d_j}{(\sum_k d_k)-1}$. Thus, ignoring the possibility of a loop, the expected change in the number of unexplored open halfedges is $\sum_{j \neq i} \frac{d_j(d_j-2)}{(\sum_k d_k)-1}$. More generally, ignoring the possibility of edges to the explored tree (which is a reasonable approximation when the tree is small), and
letting $U$ be the unexplored vertices not yet in the tree, at each iteration the expected change in the number of unexplored 
open halfedges is $\sum_{k \in U} \frac{d_k(d_k-2)}{\sum_{k \in U} d_k}$.

It turns out that there is a giant component in $G(\cD_{\ell})$ if this expected change stays positive long enough, and that in this case the giant component also contains a positive proportion of all the edges. 
Specifically, we say that $\{\cD_{\ell}\}_{\ell\geq 1}$ has a {\em gigantic component} if there exists $\epsilon>0$ such that the probability that $G({\cal D}_{\ell})$  has a component containing at least $\epsilon m(\ell)$ edges goes to 1 as $\ell$ goes to infinity. Then, recalling that the entries of $\cD=(d_1,\ldots,d_n)$ are non-decreasing, and letting 
$$j_{\D} := \min\{j \in [n] : \sum_{i=1}^j d_i(d_i - 2) > 0 ~\text{or}~j=n\}$$ 
and 
\[
R_{\D} := \sum_{i = j_{\D}}^n d_i\, ,
\]
Theorems 1 and 2 in \cite{joos2018determine} imply the following result.

\begin{thm}[{\cite{joos2018determine}}]
\label{whenhasgiantthm} 
 Let $(\cD_{\ell})_{\ell \ge 1}$ be a sequence of degree sequences each satisfying $n_0=0$. The 
 following are equivalent:
 \begin{enumerate}
 \item $G(\cD_{\ell})$ has a giant component.
 \item $G(\cD_{\ell})$ has a gigantic component.
 \item $m^{\neq 2}=\omega(1)$  and  for some $\rho>0$, and for all large enough $\ell$, $R_{\D_{\ell}} \geq \rho m^{\neq 2}$.
 \end{enumerate}
\end{thm}

\noindent {\bf Remark:} The parameter $R_{\cD}$ of degree sequence $\cD$ plays an important role later in the paper.

We note that a result analogous to Theorem~\ref{whenhasgiantthm} but for the configuration model could be proved via a reasonably straightforward application of results on 
random walks.
This fact could be used to deduce the result for $\cD_{\ell}=(d_{\ell,1},\ldots,d_{\ell,n(\ell)})$ provided that $\sum_{i \in [n(\ell)]} d_{\ell,i}^2=O(n(\ell))$ as $\ell \to \infty$, since this is a necessary and sufficient condition for the probability that the configuration model generates a simple graph to be bounded away from zero \cite{MR2497380,MR3317354}.
  The difficulty in proving the result  in general is the presence of high-degree vertices, which make loops and parallel edges more likely and also make it more difficult to prove concentration results for the relevant random walks. 
Furthermore, as we have seen, the presence of high degree  vertices weakens the bounds one can obtain on the mixing time.

\subsection{Related work} 

Much is known about diameter and mixing time of the giant component for other random graph models; we do not attempt an exhaustive survey of the literature. In the Erd\H{o}s--R\'{e}nyi random graph $G(n,p)$, Hildebrand \cite{hildebrand1996random} showed a mixing time and diameter of $\frac{\log n}{\log d}$ for $d$ being the expected degree $p(n-1)$ in the regime where $G(n,p)$ is connected whp ($p \geq \frac{\log^2 n}{n}$). Using the tools they developed in \cite{fountoulakis2007bottlenecks} and which we describe later in this section, Fountoulakis and Reed \cite{fountoulakis2008evolution} bound the mixing time of the giant component in the regime $\frac{1+\epsilon}{n} \leq p = O\left(\frac{\sqrt{\log n}}{n}\right)$ where the graph is no longer connected whp; their result is that the mixing time is $\Theta\left(\left(\frac{\log n}{d}\right)^2\right)$ where $d$ is at most $O(\sqrt{\log n})$. In particular, they identify the presence of long bare paths (containing at least $\frac{\log n}{d}$ many vertices of degree 2) as a local obstruction. 
Independently, Benjamini, Kozma, and Wormald \cite{benjamini2014mixing} obtained matching bounds of $\Theta(\log^2 n)$ in the regime of constant $d$ but use a more geometric approach that gives results on $G(n,m)$. The mixing time of the largest connected component of the  critical random graph $G(n,1/n)$ was determined in~\cite{nachmias2008critical}.

For the configuration model, the results of \cite{benjamini2014mixing} show that a configuration with given degrees is a good expander when all degrees are at least 3, leading to the natural diameter of $\Theta(\log n)$ and mixing time of $\Theta(\log^2 n)$. For i.i.d degrees following a power-law distribution, the diameter of the entire configuration has been characterized by \cite{van2005distances1, van2005distances2, fernholz2007diameter} where, e.g., the diameter can be as small as $\Theta(\log \log n)$. Our work bridges these regimes and establishes that for arbitrary supercritical degree sequences, the density of degree-two vertices governs both the diameter and the mixing time of the giant component.

\newpage
\subsection{Our approach}

\subsubsection{Switching}

To handle high-degree vertices, we use arguments that exploit a technique known as {\em switching} (which we also end up needing in the treatment of the low-degree vertices). Given two ordered edges $ab$ and $xy$ of a graph $H$ such that $b \neq x$, $a \neq y$, and $xb,ay \not \in E(H)$, {\em switching on ${ab,xy}$ in $H$} creates a new graph $H_{xy,ab}$ 
with $V(H_{xy,ab})=V(H)$ and $E(H_{xy,ab})
=(H\setminus\{xy,ab\}) \cup \{xb,ay\}$, in which all vertices have the same degrees as they do in $H$. We note that switching on
the ordered pair of edges $xb,ay$ in $H_{xy,ab}$ yields $H$. We also note that $H_{yx,ba}=H_{xy,ab}$, which implies that there are either zero or exactly two switchings between any two graphs. 

We note that if ${\mathcal A}$ and ${\mathcal B}$ are disjoint sets of graphs with degree sequence $\cD$ such that (i) 
for every $G \in {\mathcal A}$ there  are at least  $\delta({\mathcal A})$ switchings on $G$ which yield a graph in 
${\mathcal B}$ and (ii) 
for every $G \in {\mathcal B}$ there   are at most  $\Delta({\mathcal B})$ switchings on $G$ which yield a graph in 
${\mathcal A}$, then 
\begin{equation}\label{switchingineq}
    |{\mathcal A}| \le \frac{\Delta({\mathcal B})}{\delta({\mathcal A})} |{\mathcal B}|\ .
\end{equation}
This fact will be an important tool in bounding probabilities in $G(\cD)$. 

To illustrate the switching approach, we show now that the $O(n^{1/3}h(n))$ bound of Theorem \ref{thm:mainnew1} for the order of the second largest component is tight up to the $h(n)$ term, even 
for graphs with $m^{\neq 2}=m$.
In the example showing this, we take some real $\rho<10^{-6}$. For $a=a(\ell)>10$, let ${\mathcal D}_{\ell}=(d_1,\ldots,d_{\ell})$ have one vertex of degree  $\lceil 2 \rho\ell \rceil $ and a set $H$ of $\lfloor\ell^{1/3} \rfloor$ vertices of degree $\lfloor \frac{\ell^{1/3}}{a} \rfloor $. All other vertices have degree 1.  We know there is a giant component without relying on Theorem \ref{whenhasgiantthm} as $G({\mathcal D}_{\ell})$ always has a component   containing $n=\ell$  and at least 
 $2\rho n$ other vertices. 
To show the bound from Theorem \ref{thm:mainnew1} is tight, it suffices to prove that the probability that some vertex of $H$ is not in this
component, but instead is the center of a star component, is at least $1-\frac{1}{a}$.  
To this end, we note that $2m=(1+2\rho+o(1))\ell$ and that in any realization of $G=G(\cD)$ almost all vertices lie in edge components.
In particular there are always at least $\frac{m}{3}$ edge components. So for $x,y \in H$, given a graph $G$ with $xy \in E(G)$, we can switch  $xy$ with at least $\frac{2m}{3}$ ordered edges  to obtain a graph where $xy \not\in E(G)$. On the other hand, to switch 
from a graph in which $xy \not\in E(G)$ to a graph with $xy \in E(G)$,
we need to use an edge incident to $x$ and an edge incident to $y$ so there are fewer than  $\frac{m^{2/3}}{a^2}$ choices. Thus, taking an auxiliary bipartite graph where $\cA = \{G \in \cG(\cD) : xy \in G\}$ and $\cB = \{G \in \cG(\cD) : xy \notin G\}$, we apply \eqref{switchingineq} to obtain $|\cA| \leq \frac{m^{2/3}/a^2}{2m/3} |\cB|$. But note that since we are working with a uniform distribution we then have $\p{xy \in E(G)} = \frac{|\cA|}{|\cA| + |\cB|} \leq \frac{|\cA|}{|\cB|} \leq \frac{m^{2/3}/a^2}{2m/3} < \frac{2}{a^2\ell^{1/3}}$.

It follows that the expected number of edges within $H$ is at most 
${\ell^{1/3} \choose 2}\frac{3}{2a^2l^{1/3}}<\frac{\ell^{1/3}}{a^2}$. 
Since $a>10$, applying Markov's inequality, the  probability that there are at most $\frac{\ell^{1/3}}{5}$ such edges, and hence that at least $\frac{\ell^{1/3}}{2}$ vertices of $H$ are adjacent to no other vertex 
of $H$, is at least $1-\frac{1}{2a}$. 

Next for each $i \ge \frac{\ell^{1/3}}{2}$ we upper-bound the conditional probability, given that there are exactly $i$ vertices of $H$ adjacent to no other vertex of $H$, that $n$ is adjacent to all of them. In this situation, our auxiliary bipartite graph is formed between the set $\cA_i$ of graphs satisfying this property and $\cB_i$ of graphs in which $n$ is not adjacent to exactly one of the $i$ vertices of $H$ which have no neighbors in $H$.
Given $G \in \cA_i$, we can  switch any edge  from $n$ to $H$ with at least $\frac{2m}{3}$ ordered edges to obtain some graph in $\cB_i$ so there are at least $\frac{2mi}{3}>\frac{m\ell^{1/3}}{3}$ such switches. On the other hand, to switch from some $G \in \cB_i$ to an element of $\cA_i$,  we must use an ordered edge incident to $n$ and an ordered edge incident to the unique vertex in the set of $i$ vertices which $n$ does not neighbor, so there are fewer than $\frac{4 \rho m\ell^{1/3}}{a}$
choices. Again by \eqref{switchingineq},  the conditional probability we are bounding is at most $\frac{4\rho m\ell^{1/3}/a}{m\ell^{1/3}/3} < \frac{1}{2a}$ and the proof is complete.

The above example illustrates one main way in which we use the switching inequality. Another way we will use it is the following. To show that $G(\cD)$ is likely to have some property $Q$, partition $\cG(\cD)$ into sets $\cF_1, \cF_2, \dots, \cF_k$ such that $\bar{Q} = \bigcup_{i \geq i_0} \cF_i$ and $\p{\bar{Q}} = \frac{\sum_{i = i_0}^k |\cF_i|}{\sum_{i = 1}^k |\cF_i|}$. Using an auxiliary bipartite graph on $\cF_{i}$ and $\cF_{i+1}$, we apply the switching inequality to show that $\frac{|\cF_{i+1}|}{|\cF_i|}$ is bounded above by some $C < 1$, and so telescoping gives us a bound of $\sum_{i = i_0}^k |\cF_i| \leq \sum_{i = i_0}^k C^{i-1}|\cF_1|$. In subsequent arguments, we will often omit the precise details of the auxiliary bipartite graph and of such ratio computations, and simply describe the bounds on the number of switches.

\subsubsection{Bounding mixing time and diameter  via conductance}

For a graph $H=(V,E)$ and a set $S \subseteq V$, 
we let $d(S)=d_H(S)$ be the sum of the degrees in $H$ of the vertices in $S$.
We recall that the probability that the lazy random walk is in $S$ when at its steady state, denoted $\pi(S)=\pi_H(S)$, is $\frac{d(S)}{2m}$.   
We let $\ex(S)=\ex_H(S)$, the {\it excess of $S$}, be $d(S)-2|S|+2$. 
We let $\out(S)=\out_H(S)$ be the number of edges between $S$ and $V-S$. 
We let $\cond(S)=\cond_H(S)=\frac{\out(S)}{d(S)}$.  For $0<x<d(V)$
we let $\cond(x)=\cond_H(x)$ be the minimum of $\cond(S)$ over all connected $S$ such that $\frac{x}{2} \le d(S) \le \min\{x,d(V)/2\}$. We omit the underlying graph $H$ from the preceding notation whenever it is clear from context.

Our approach to bounding the {\em diameter} using conductance 
 is to apply the following, observation. 
 Write $\partial_i(v) $ for the set of vertices at distance $i$ from $v$
    and $B_i(v)$ for the set of vertices at distance at most $i$ from $v$.

\begin{prop}
\label{conductanceboundnew}
    For any graph $H$, the diameter  of  $H$ is  at most 
    $2\sum_{j=1}^{\lceil \log d(V) \rceil-1} \cond(2^{j})^{-1}$.     
\end{prop}
\begin{proof} 
    
    Note that $d(B_i(v))\geq d(B_{i-1}(v))+\out(B_{i-1}(v))$. So,
    setting $i_k$ to be the smallest $i$ such that $d(B_i(v))>\min\{2^k, d(V)/2\}$, if $0 \le k \le \lceil \log d(V) \rceil-2 $
    and $i$ is such that $i_k \le i<i_{k+1}$, then $d(B_{i+1}(v))>d(B_{i}(v))(1+\cond(2^{k+1}))>d(B_{i}(v))+2^k\cond(2^{k+1})$. Hence  $i_{k+1} \le i_k+\cond(2^{k+1})^{-1} $. Thus  for every $v$, $$d(B_{\sum_{j=1}^{\lceil \log d(V) \rceil-1} \cond(2^{j})^{-1}}(v))>\frac{d(V)}{2}\ .$$
    Hence for every $u$ and $v$, $B_{\sum_{j=1}^{\lceil \log d(V) \rceil-1} \cond(2^{j})^{-1}}(u)$ and 
    $B_{\sum_{j=1}^{\lceil \log d(V) \rceil-1} \cond(2^{j})^{-1}}(v)$  intersect. 
\end{proof}

Our approach to bounding the {\em mixing time} using conductance 
 is to apply the following which is a special case of the main result of \cite{fountoulakis2007bottlenecks}.

\begin{thm}[\cite{fountoulakis2007bottlenecks}]
\label{conductancebound}
    For any graph $H$,  the mixing time of the lazy random walk on $H$ is  
    $O(\sum_{j=1}^{\lceil \log  d(V) \rceil-1} \cond(2^{j})^{-2})$. 
\end{thm}

\subsubsection{Partitioning a sequence of degree sequences}
\label{sec:dichotomy}

For the remainder of the paper, we consider a sequence $(\cD_{\ell})_{\ell \ge 1}$ of degree sequences with $n_0(\cD_{\ell})=0$ and for which, for some  $\rho>0$, 
for all large enough $\ell$, $R_{\cD_{\ell}} \geq \rho m^{\neq 2}=\omega(1)$. 
By Theorem~\ref{whenhasgiantthm}, it follows that the sequence has a giant component, i.e., there  exists some $\epsilon=\epsilon(\rho)>0$ such that with high probability $G(\cD_{\ell})$ contains a component of order at least $\epsilon n(\ell)$.
  Since if this property holds for some $\rho$, it holds for all smaller $\rho$, we can and do take this $\rho$ to be smaller than some fixed $\rho_0>0$, which we choose small enough to make various inequalities that arise during the proofs work out. 

We apply different techniques according to the shape of the degree sequences
in our sequence, so our first step is to divide the sequence of degree sequences into a set of three subsequences 
of degree sequences. To this end, for any $\mu>0$, we say a degree sequence $\cD$ {\it has a $\mu$-center} if  it holds that 
$\sum_{n-d_n+1}^n d_i\geq \mu^2 m^{\neq 2}$ and $m^{\neq 2} \geq \mu^3 m$. 

Given $\rho$ as above, we choose  $\mu>0$ sufficiently small in terms of $\rho$, and in particular less than $\frac{\epsilon}{8}$, and for some $\lambda$ sufficiently small in terms of $\mu$ (and hence $\rho$) 
we divide our sequence of degree sequences into the following three subsequences:

\begin{enumerate}
    \item[(i)] Degree sequences with a $\mu$-center for which $\sum_{i : d_i>\lambda \sqrt{m}} d_i \le \rho \mu^{10}m$,
    \item[(ii)] Degree sequences with  a $\mu$-center for which $\sum_{i : d_i>\lambda \sqrt{m}} d_i > \rho \mu^{10}m$,
    \item[(iii)] Degree sequences without   a $\mu$-center. 
\end{enumerate}

We note that if 
$\Delta < \mu^3 \sqrt{m}$ then $\sum_{i=n-d_n+1}^n d_i \le \Delta^2 < \mu^6 m$, so if $m^{\ne 2} \ge \mu^3 m$ then $\sum_{i=n-d_n+1}^n d_i < \mu^3 m^{\ne 2}$. Similarly, if $m \ge h(m)m^{\neq 2}$ for some $h(m) \to \infty$, the condition $m^{\neq 2} \ge \mu^3 m$ is trivially violated for large $m$. Therefore in both cases, the degree sequence has no $\mu$-center. By applying our results which bound mixing time and diameter using conductance, we shall prove  the following, which implies Theorem \ref{thm:mainnew1} for degree sequences without a $\mu$-center along with Theorems \ref{thm:mainnew2} 
and \ref{thm:mainnew4}, whose hypotheses restrict to sequences without a $\mu$-center.

\begin{thm}\label{thm:mainnew5} For every $\rho \le \rho_0$ 
there is a $\mu>0$ such that the following holds. 
Let $(\cD_{\ell})_{\ell \ge 1}$ be a sequence of feasible degree sequences 
such that for each $\ell$, $n_0(\cD_{\ell})=0$, $m^{\neq 2}=\omega(1)$, $R_{\cD_{\ell}} \ge \rho m^{\neq 2}$ and  ${\cD}_{\ell}$ is without a $\mu$-center.
    Then with high probability all  components except the largest have  $O(\frac{m \log m^{\neq 2}}{m^{\neq 2}})$ vertices  
    and every component has diameter
     $O(\frac{m \log m^{\neq 2}}{m^{\neq 2}})$ and mixing time $O((\frac{m \log m^{\neq 2}}{m^{\neq 2}})^2)$. 
\end{thm}

 In Section \ref{sec:withcenter} we prove Theorem \ref{thm:mainnew1} for degree sequences with a $\mu$-center by combining Theorem \ref{thm:mainnew5} and ad-hoc techniques using bounded-distance expansions from high-degree vertices.  Specifically, we prove and exploit the fact that when the degree sequence has a center, whp the set of vertices of degree at least $\lambda \sqrt{m}$ have pairwise distance 
$o(\log m)$ from each other.

We note that when applying Theorem \ref{thm:mainnew5}, we actually will apply the following equivalent statement: for every $g$ going to $\infty$ as $m \to \infty$ and  $\rho \le \rho_0$, 
there is a $\mu>0$, a function $b$ going to $0$ as $m \to \infty$, and a function $f$ with 
$f(x)=O(x)$ as $x \to \infty$
 such that the following holds. 
For every degree sequence without a $\mu$-center with  $n_0(\cD)=0$, $m^{\neq 2} \ge g(m)$, and $R_{\cD} \ge \rho m^{\neq 2}$, the probability that the following fails is at most $b(m)$: 
  all  components except the largest have at most $f(\frac{m \log m^{\neq 2}}{m^{\neq 2}})$ vertices   
    and every component has diameter
     at most $f(\frac{m \log m^{\neq 2}}{m^{\neq 2}})$ and mixing time at most $f(\frac{m \log m^{\neq 2}}{m^{\neq 2}})^2$.

\subsection{Notation and outline}
We perform several reductions to $G$ in order to analyze it, and introduce some notational conventions to aid the reader. Typically $G$ will refer to $G(\D)$, a uniformly random simple graph with degree sequence $\D$. $H$ will be used for a generic graph or as a dummy variable. $J$ will usually be a colored multigraph that results from our first reduction applied to $G$ (described in Section \ref{subsec:coloredconfigmodel}). $K$ will be a multigraph with minimum degree 3 that results from our second reduction applied to $J$ (described in Section \ref{subsection:kernel}). When referring to a graph, the size will mean the number of edges while the order will mean the number of vertices; when referring to subsets, we will use size and order interchangeably.

The outline of the remainder of the paper is as follows. In \Cref{sec:degreetwo}, we characterize the distribution of degree two vertices between cycle and non-cycle components and within paths of the non-cycle components. In \Cref{sec:nocenter}, we address degree sequences without a center, proving \Cref{thm:mainnew5}. In \Cref{sec:withcenter}, we handle degree sequences with a center, finishing the proof of \Cref{thm:mainnew1}.

\section{Vertices of Degree Two}\label{sec:degreetwo}

Theorem \ref{whenhasgiantthm} implies that if $m^{\ne 2}=m^{\neq 2}_{\cD_{\ell}}=O(1)$  then $\{\cD_{\ell}\}_{\ell \ge 1}$ does not have a giant  component. Throughout the section, we assume the sequence of degree sequences does have a giant component, so, in particular, we can and do assume there is a function $f$ such that $f(\ell) \to \infty$ as $\ell \to \infty$ and such that $m^{\neq 2}_{\cD_{\ell}} \ge f(\ell)$. 

 We use $\Cyc(G)$ to denote the subgraph of $G$ consisting of all cycle components in $G$ and $\cyc(G)$  to denote  the number of edges (or, equivalently, vertices) contained by $\Cyc(G)$.  
 
 \subsection{Bounding the size of the union of the cycle components}
 Our goal in this subsection is to prove the following proposition. 
 \begin{prop}\label{fewcycles}
      There exists $f_0>0$ such that for 
       any degree sequence $\mathcal D$   satisfying $m^{ \neq 2}>f_0$, we have 
       $$\cyc(G(\cD)) < \frac{203(n_2 +m^{\neq 2}) \log m^{\neq 2}}{m^{\neq 2}}$$ 
       with probability at least $1-(\log \log m^{\neq 2})^{-1}$.
 \end{prop}

To prove this, we use the following result on the number $C_t$  of graphs which are disjoint unions of cycles on 
$t$ labelled vertices.

\begin{thm}[\cite{flajolet2009enumeration}, Example VI.2]
$$C_{t}=\left(1+\frac{5}{8t}+O(t^{-2})\right)\frac{e^{-3/4}}{\sqrt{\pi t}}t!$$
\end{thm}
In particular, this implies:
\begin{cor}
\label{cyclecountcor}
    $$\frac{C_{t+1}}{C_t}=\left(1+\frac{1}{2t}+o\Big(\frac{1}{t}\Big)\right)t\ .$$
\end{cor}

\begin{proof}[Proof of Proposition~\ref{fewcycles}]
Given a multigraph $\mathcal H$ which has  no vertices of degree 2, we write $\cG_{\mathcal H}=\cG_{\mathcal H}(\cD)$ for the set of elements of $\cG(\cD)$ which consist of a subdivision of $\mathcal H$
together with some cycle components. We note that if $\cG_{\cH} \ne \emptyset$ we must have $e({\mathcal H})=m^{\neq 2}$.
We will control the number of edges in cycle components of a uniform sample from $\cG_{\mathcal H}$ using switching. However, we use a slightly different form of switching, which involves taking a vertex of degree 2 from a component and using it to subdivide an edge in a different component. 
 
 More precisely, let $\cG_{{\mathcal H},t}=\cG_{{\mathcal H},t}(\cD)$ be the graphs in $\cG(\cD)$ consisting of a subdivision of $\mathcal H$ along with cycle components that together contain exactly $t$ vertices. We swap between graphs in $\cG_{{\mathcal H},t}$ and $\cG_{{\mathcal H},t+1}$. Given $G \in \cG_{{\mathcal H},t+1}$, let $G_{\mathcal H}$ be its component that is a subdivision of $\mathcal H$. We produce $G' \in \cG_{{\mathcal H},t}$ by taking $G_{\mathcal H}$, choosing a vertex from the cycle components of $G$ and using it to subdivide an edge of $G_{\mathcal H}$, and then distributing the remaining $t$ vertices across cycle components. This yields at least $(t+1)(n_2-(t+1)+e(\cH))C_t$ choices for $G'$. 

In the other direction, given $G' \in \cG_{{\mathcal H},t}$, we construct $G \in \cG_{{\mathcal H},t+1}$ by taking $G'_{\mathcal H}$, suppressing one of its vertices of degree 2, and distributing the $t+1$ remaining vertices amongst cycle components. This yields at most $(n_2-t)C_{t+1}$ choices.
By the same reasoning as \eqref{switchingineq}, we see that 
\begin{equation}\label{eq:cycswitch}
    |\cG_{{\mathcal H},t+1}| \leq \frac{n_2 - t}{(t+1)(n_2 - t -1+ e({\mathcal H}))}\frac{C_{t+1}}{C_t} |\cG_{{\mathcal H},t}|\ .
\end{equation}

Applying Corollary \ref{cyclecountcor}, we may fix $t_0$ such that $\frac{C_{t+1}}{(t+1)C_t} \le 1$ for $t \ge t_0$. Then for $t \ge t_0$ we have 
\begin{align*}
    |\cG_{{\mathcal H},t+1}|
    & \le \left(1-\frac{e({\mathcal H})-1}{n_2-t-1+e({\mathcal H})}\right)|\cG_{\mathcal{H},t}|\, ,\\
    & \le \exp\left(-\frac{e({\mathcal H})-1}{n_2-t-1+e({\mathcal H})}\right)|\cG_{{\mathcal H},t}|\, ,
\end{align*}
so by a telescoping product we obtain that for $n_2 \ge   s > t \ge t_0$, 
\begin{align*}
    |\cG_{{\mathcal H},s}|
    & \le \exp\left(- \sum_{i=t+1}^s \frac{e({\mathcal H})-1}{n_2+e({\mathcal H})-i}\right)|\cG_{{\mathcal H},t}|\, ,\\
    & \le \left(\frac{n_2+e({\mathcal H})-s}{n_2+e({\mathcal H})-t}\right)^{(e({\mathcal H})-1)}|\cG_{{\mathcal H},t}|\\
    & \le 
    \exp\left(-\frac{e({\mathcal H})(s-t)}{2(n_2+e({\mathcal H})-t)}\right)|\cG_{{\mathcal H},t}|. 
\end{align*}

If $n_2 \ge e({\mathcal H})/100$ then 
whenever $s-t \ge 202 n_2\log e({\mathcal H}) /e({\mathcal H})$ this gives 
\[
|\cG_{{\mathcal H},s}| \le e({\mathcal H})^{-1}|\cG_{{\mathcal H},t}|.
\]
It follows that if $G$ is a uniform sample from $\cG_{\mathcal H}$,  then $\cyc(G)< \frac{203n_2\log m^{\neq 2}}{m^{\neq 2}}$ with  probability  at least $1-\frac{2}{m^{\neq 2}}$. 

On the other hand, if $n_2 \le e({\mathcal H})/100$ then $e({\mathcal H})/2(n_2+e({\mathcal H})-t) \ge 1/3$ so the above bound gives 
\[
|\cG_{{\mathcal H},s}| \le e({\mathcal H})^{-1}|\cG_{{\mathcal H},t}| 
\]
whenever $s-t \ge 3\log e({\mathcal H})$, so if $G$ is a uniform sample from $\cG_{\mathcal H}$, then with probability  at least  $1-\frac{2}{m^{\neq 2}}$, we have $\cyc(G)<3\log e({\mathcal H}) $, which is less than $\frac{203(n_2+m^{\neq 2})  \log m^{\neq 2}}{m^{\neq 2}}$.

In both cases,  we see that for $G$ a uniform sample from $\cG_{\mathcal H}$, we have  
\[
\p{\cyc(G)<\frac{203(n_2+m^{\neq 2}) \log m^{\neq 2}}{m^{\neq 2}}} > 1-\frac{1}{\log \log m^{\neq 2}}.
\]
Given $\mathcal H$, since
$G(\cD)$ conditioned on being in $\cG_{\mathcal H}$ is a uniformly random element of $\cG_{\mathcal H}$, it follows that the probability that  $\cyc(G(\cD)) <\frac{203(n_2+m^{\neq 2}) \log m^{\neq 2}}{m^{\neq 2}}$ is  also at least  $1-\frac{1}{\log \log m^{\neq 2}}$.
\end{proof}
When $n_2=\omega(m^{\neq 2})$, we sometimes condition on the event that the number $n'_2$ of degree-two vertices in the subdivision is $(1+o(1))(n_2)$. Proposition~\ref{fewcycles} tells us that this occurs with probability $1-o(1)$.

\subsection{Cores, kernels, homeomorphic reductions, and coloured reductions} 

Given a simple graph $G$, the {\em homeomorphic reduction} of $G$ is the multigraph $J(G)$ obtained from $G$ as follows. First delete all cycle components. Then, for each maximal-length path $P$ in $G-\Cyc(G)$ all of whose internal vertices have degree $2$, replace $P$ by a single edge joining the endpoints of $P$. 

Recall that for a graph $G$, the {\em core} of $G$ is the maximum subgraph of $G$ with minimum degree $2$. The {\em kernel} of $G$, denoted $K(G)$, is the homeomorphic reduction of the core of $G$.

The {\em coloured reduction} of $G$ is obtained by colouring the edges of the homeomorphic reduction $J=J(G)$ in green, yellow, and red, according to the following rules. We colour an edge of $J$ red if the corresponding path in $G$ has no internal vertex, 
yellow if it has just one internal vertex, and green if it has at least 2 internal vertices. We note that by our assumption that $G$ is simple, no two red edges of $J$ can be parallel and no loop can be yellow or red. 

Given a degree sequence $\cD=(d_1,\ldots,d_n)$, we write $\cJ(\cD)=\{J(G): G \in \cG(\cD)\}$. 
Note that if $G$ has degree sequence $\cD$ then $J(G)$  
has $n-n_2(\cD)$ vertices and $m^{\neq 2}=m(G)-n_2(\cD)$ edges, and the degree sequence of $J$ is $\cD$ with the vertices of degree 2 deleted, which we denote by $\cD^{\ne 2}$.

\subsection{Switching on the coloured reduction and on the kernel}\label{subsec:coloredconfigmodel}

{Fix a simple graph $G$ and write $(r_G,y_G,g_G)$ for the triple which lists the number of red, yellow and green edges of $J(G)$. 
If $m^{\neq 2}< \frac{m}{100}$ then in obtaining bounds on the conductance   we will use a {\em coloured switching operation}, described in the next  two paragraphs,   which 
switches between graphs in the family 
\[
\cF_{r,y,g}=\cF_{r,y,g}(\cD)=\{G\in \cG(\cD)~:~(r_G,y_G,g_G)=(r,y,g)\}
\]}
 for a given $r, y, g$.

Recall from Section~\ref{sec:configmodel} that for a degree sequence $\cD=(d_1,\ldots,d_n)$, $\cM(\cD)$ is the set of matchings generating simple graphs with degree sequence $\cD$.  Starting from any matching $M \in \cM(\cD)$, letting $G$ be the graph corresponding to $M$, we then define a matching $M^*$ of the (uncoloured) halfedges incident to vertices of the multigraph $J(G)$, as follows. If edge $e=uv$ of $J$ corresponds to a path $P$ in $G$, then we represent $e$ as a matching of the unique pair of halfedges $hh'$ along $P$ that are incident to $u$ and $v$, respectively (so $h$ and $h'$ are paired in $M^*$). Note that $h$ and $h'$ are also paired in $M$ if and only if $e$ is red. Note also that the halfedges of $M^*$ are precisely those halfedges incident to vertices whose degree is not $2$. 
The (coloured) multigraph $J$ may be recovered from the matching $M^*$, together with the partition $\Pi=(\Pi_r,\Pi_y,\Pi_g)$ of the pairs of the matching into $r$ that are red, $y$ that are yellow and $g$ that are green in such a way that  there are no parallel edges which are both red, or loops which are red or yellow.

In {\em coloured switching}, we perform switching on the ordered paths of $G$ corresponding to the oriented edges of $J$ rather than on the  oriented edges of $G$ itself. The switching corresponding to a pair $xy,uv$ of oriented edges of $J$
is accomplished by (i) deleting the ordered  $x-y$ and $u-v$ paths   and  (ii) adding ordered $x-v$ and $u-y$ paths  
  so that the interior of the $x-v$  path is the interior of the old $x-y$  path and the interior of the 
  $u-y$ path is the interior of the old $u-v$ path. Thus, the colour of the new edge from $x$ to $v$ is the same as that of the deleted  edge from $x$ to $y$ and the colour of  the new edge between $u$ and $y$ is the same as the colour of the deleted edge between $u$ and $v$. We only consider switchings that preserve the property that the underlying graph is simple, i.e. we must maintain the property that there are no two parallel red edges and no yellow or red loop. Note that a coloured switching yields a new graph $G'$ with $(r_{G'},y_{G'},g_{G'})=(r_G,y_G,g_G)$; in other words, if $G \in \cF_{r,y,g}$ then also $G' \in \cF_{r,y,g}$. Finally we note that  a switching on the oriented 
  edges $xy$ and $uv$ of $J$ is equivalent to switching the oriented final edges of the corresponding 
  paths in $G$.  

  At certain points in the argument, we will restrict to switching only on the ordered paths of $G$ corresponding to the oriented edges of the kernel $K(G)$. Since $K(G)$ is the homeomorphic reduction of the core, these switchings work identically to those described in the preceding paragraph. For both the family of switchings described in this paragraph, and that described in the previous paragraph, we are actually carrying out switchings in $G$ and  the logic leading to the inequality \eqref{switchingineq} applies. It follows that that inequality can be applied to bound probabilities of $G$, of $J(G)$ or of $K(G)$ having given properties even when restricting the set of switchings we consider. 

We can carry out an exploration generating the multigraph $J(G)$ using the same approach we sketched for exploring $G$ in Section~\ref{sec:configmodel}. 
When exposing the edges out of a vertex $v$, we do so halfedge by halfedge. We will describe this exploration in more detail later; we first explain how to bound the typical lengths of the paths in $G$ corresponding to green edges of $J(G)$ when $G$ is a uniformly random element of $\cF_{r,y,g}$, which we will write as $G \in_u \cF_{r,y,g}$.

Having fixed $r,y,g$, 
we expose the internal (degree-$2$) vertices on paths corresponding to green edges as follows. In what follows, we write $N = n_2 - 2g - y$; so the number of vertices on the paths corresponding to green edges is $N+2g$. We arbitrarily order the  green edges and orient each of them. We reveal the indices  of the 
first and last vertex on each green path. 
To specify the remainder of the paths, 
we choose a random permutation of $g-1$ delimiters and the  $N$ remaining vertices of  
degree 2 lying on these paths.  We add a delimiter at the beginning and the end of the permutation.
The vertices between the $i^{\mathrm{th}}$ and $(i+1)^{\mathrm{st}}$ delimiter are placed on the $i^{\mathrm{th}}$ green edge in the given order.

Equivalently, we can think of drawing without replacement from a bin of $N$ vertices and $g-1$ delimiters; then the probability that we draw exactly $k$ vertices followed by a delimiter is  
\begin{align*}\frac{N}{N+g-1}\left(\frac{N-1}{N+g-2}\right) \cdots \left(\frac{N-k+1}{N+g-k}\right) \left(\frac{g-1}{N+g-1-k}\right) \\
= \frac{g-1}{N+g-k-1}
{\prod_{j=1}^k \left(1-\frac{g-1}{N+g-j}\right)}\ .
\end{align*}

We say the {\em length} of an edge of $J$ is the number of edges on the corresponding path in $G$. Using the perspective above, we characterize the distribution of edge lengths via the following.

\begin{prop}
\label{ballsinbins}
    Let $s, B \in \mathbb N$. Conditioned on the set of yellow, red, and green edges,
    for any set $S$ of green edges with $|S| = s$, the probability that 
    the sum of the lengths of the   edges in $S$ exceeds $2s+B$
    is at most $\p{\mathrm{Bin}(g-1, \frac{B}{N})<s}$.
\end{prop} 
    \begin{proof}
    In the process we have described for determining the length 
    of the green edges, we can and do use the vertices of $S$ as the first $s$ edges.  Thus the probability that the first $s$ edges have total length greater than $2s+B$ is the 
    probability that fewer than $s$ delimiters are chosen in the first
    $B$ draws. We can determine when the delimiters are drawn by exposing the draw for each delimiter in turn, with all remaining draws
    equally likely. Conditioned on the previous choices, the probability that the  $i^{\mathrm{th}}$ delimiter is drawn in the first $B$ draws 
    is at most $\frac{B}{N+g-i}$. Since $\frac{B}{N+g-i} \leq \frac{B}{N}$, this proves the claim.
    \end{proof}

Using $g_i=g_i(G)$ to denote the number of green edges of $J(G)$ of length $i$,
the following proposition bounds the likely lengths of typical green edges, and also the number $g_3$ of green edges of length three. 

\begin{prop}
\label{greenpathsprop}
For any $0 < \delta \le \frac{1}{100}$ there is $n_0\in \N$  such that the following holds.  
Assume $\cD$  is a degree sequence of length $n \ge n_0$ such that $m^{\ne 2} \le \delta m$.  Fix $(r,y,g)$ with $g > 1$, $\cF_{r,y,g}(\cD)$ non-empty and let $G \in_u \cF_{r,y,g}$. Then for any $c>0$, conditional on  $J(G)$, the probability the path corresponding to a fixed green edge in $J(G)$ has at least $2+\frac{cN}{g-1}$ internal vertices is at most  $e^{-c/2}$. Moreover, $\p{g_3(G) \le \frac{10\delta}{9-36\delta}g} \geq 1-n^{-1/7}$.  
\end{prop}

We will apply this proposition to the two cases of $\delta = \frac{1}{100}$ and $\delta = \mu^3$ where $\mu < \frac{1}{12}$ is small enough in terms of $\rho$. In the former case, we can conclude that $g_3 \leq \frac{g}{80}$ whp and in the latter case, $g_3 \leq \frac{\mu^2m^{\neq 2}}{3}$ whp.

\begin{proof}
Observe that if $m - n_2 = m^{\neq 2} \leq \delta m$, then 
\begin{itemize}
    \item $n_2 \geq (1-\delta)m \geq (1/\delta - 1)m^{\neq 2}$, 
    \item $y \leq m^{\neq 2} \leq \delta m$, and
    \item $g \leq m^{\neq 2} \leq \delta m$.
\end{itemize}

We thus have
\begin{align*} N = n_2 - 2g - y &\geq \min\{(1-4\delta) m, (1/\delta - 4)m^{\neq 2}\}\\
&\geq \min\{(1-4\delta) m,(1/\delta - 4)g\}\\
&\ge \min\left\{\frac{96m}{100}, 96g\right\}\ .
\end{align*}

Now, by symmetry we can replace any fixed green edge by the first green edge in the permutation. 
The probability that this edge has at least $k+2$ internal vertices is 
\[ \prod_{j=1}^k\left(1-\frac{g-1}{N+g-j}\right)
\le  \left(1-\frac{g-1}{N+g}\right)^{k}
\le e^{-k(g-1)/(N+g)} \le e^{-96k(g-1)/(97N)}. 
\]
It  follows that the probability this green edge  has at least $2+\frac{cN}{g-1}$ internal vertices is  
at most $e^{-c/2}$.

The probability a green edge has exactly 2 internal vertices is $\frac{g-1}{N+g-1} < \frac{g}{N}$. Thus $\E{g_3} = \frac{g(g-1)}{N+g-1} \le \frac{\delta}{1-4\delta}g$. 

If $g<4N^{2/3}$  then  since $N \ge \frac{96m}{100}>\frac{n}{3}$
and $\E{g_3}<\frac{4g}{N^{1/3}}$, for large enough $n_0$,
 Markov's inequality tells us that $\p{g_3<\frac{10\delta}{9-36\delta}g} \geq 1-n^{-1/7}$.  Otherwise, $g \geq 4N^{2/3}$ and hence $\e[g_3]>4N^{1/3}>n^{1/3}$. Thus,  $\e[g_3]<\frac{\e[g_3]^2}{n^{1/3}}$. 

We now compute the expected number of ordered pairs of green paths both of which have two internal vertices. 
By symmetry, this is $g(g-1)p+\e[g_3]$ where $p$ is the probability the 1st and 2nd paths of the permutation each have two internal vertices. So, the first term of our sum is 
$$g(g-1) \frac{g-1}{N+g-1} \cdot \frac{g-2}{N+g-2}$$

 Using that $g \ge 2$ and $N \ge (1/\delta - 4)g > 4$, it follows straightforwardly that
$\frac{g-2}{N+g-2} \leq \left(1 + \frac{4}{N}\right) \frac{g-1}{N+g-1}$, so the preceding expression is at most 
$$g(g-1)\left(\frac{g-1}{N+g-1}\right)^2 \left(1 + \frac{4}{N}\right) \le (1+\frac{4}{N})\e[g_3]^2\ .$$
Hence, since $N > n/3$, we have $\e[g^2_3] \le (1+\frac{4}{N})\e[g_3]^2 + \frac{\e[g_3]^2}{n^{1/3}} \leq \left(1+\frac{2}{n^{1/3}}\right)\e[g_3]^2$ and the second moment method  gives us the claim.
\end{proof}

The following proposition concerns the distribution of red, yellow, and green edges.
\begin{prop}
\label{redandyellowpathsprop}
    Fix $\delta \in (0,1/100)$. If $m^{\neq 2} \leq \delta m$ then, letting $r, y, g$ be the number of red, yellow, and green edges, respectively, with high probability as $m \to \infty$ the following hold:
    \begin{enumerate}
        \item $r \leq \frac{3\delta}{1-4\delta} m^{\neq 2}$, 
        \item $y \leq \frac{3\delta}{1-4\delta} m^{\neq 2}$, and
        \item $g \geq \frac{1-8\delta}{1-4\delta}m^{\neq2}.$
    \end{enumerate}
\end{prop}

\begin{proof}
We first bound the probability that the number of red edges exceeds the claimed bounds. 
By the preceding proposition, and using the standing assumption that $m^{\neq 2} = \omega(1)$, it is enough to show that $\p{r>2g_3+2\sqrt{m^{\neq 2}}+1}$ is $o(1)$. 

If $m^{\neq 2} \leq \delta m$, then as in the previous proof $n_2 \geq \left(\frac{1}{\delta} - 1\right)m^{\neq 2}$, so there are at least $\left(1 - \frac{4\delta}{1-\delta}\right)n_2$ degree two vertices on paths corresponding to (green) edges with at least 5 internal vertices.

Next, for integer $i \ge 0$, we consider 
an auxiliary bipartite graph between the set ${\cal F}_{i+1}$ of coloured multigraphs with degree sequence $\cD_n^{\ne 2}$ with $2g_3+i+1$ red edges and the
set ${\cal F}_i$ of those 
with $2g_3+i$ red edges. We place an edge between a multigraph of the former type and one of the latter 
if the second can obtained from the first by suppressing two of the vertices on green paths of length at least 6, then subdividing a red edge twice. 
Then the degree of an element in $\cal F_{i+1}$ is the number of ways to choose two such vertices on green edges of length at least 6 along with a red edge; by our earlier observation this is at least $(2g_3 + i + 1)((1 - \frac{4\delta}{1-\delta})n_2)((1 - \frac{4\delta}{1-\delta})n_2-1)$.
On the other hand, the degree of an element in $\cal F_i$ is the number of ways to choose a green edge of length at least $4$ and two edges on paths of length at least 4; we upper-bound the latter simply by the total number of edges and thus obtain $\Delta(\cal F_i) \leq g_3(n_2 + m^{\neq 2})^2$. 
By \eqref{switchingineq} and the assumption that $\delta \leq \frac{1}{100}$, we have
$\frac{|{\cal F_{i+1}}|}{|{\cal F}_i|} \le \frac{2}{3}$. 
Since $m^{\ne 2}=\omega(1)$, 
it follows that with high probability  $r \le 2g_3+2\sqrt{m^{\neq 2}}+1$, so (1) holds. 

A similar argument yields that (2) holds. We note (3) holds if (1) and (2) do. 
\end{proof}

It follows from the preceding proposition that when $\mu^3 m \leq m^{\neq 2} \le \frac{m}{100}$, whp as $m\to \infty$ we have
\begin{equation}\label{condition-ryg-bigger}
r \le \frac{m^{\neq 2}}{20},\quad 
y \le \frac{m^{\neq 2}}{20},\ \ \text{ and }\ \ 
g \ge \frac{4m^{\neq 2}}{5}\,.
\end{equation}

We can say much more when $m^{\neq 2} \le \mu^3 m$.
Given two random subsets $A,B$ of a ground set $\cX$, we say that $A$ stochastically dominates $B$ if $A$ and $B$ be can be coupled so that $B \subseteq A$. Also, we write $J^-=J^-(G)$ for the underlying uncoloured graph corresponding to $J(G)$, and for $e \in E(J(G))$, we write $J^{e-}$ for the partially coloured graph obtained from $J$ by uncolouring edge $e$.
\begin{prop}\label{specificedgeredoryellowprop}
    For all sufficiently small $\mu$, if $m^{\ne 2} \leq \mu^3 m$, then conditionally given $J^-$, the set of green edges of $J$ stochastically dominates a $\mathrm{Bin}(E(J),1-\frac{\mu^2}{3})$ subset of $E(J)$. 
\end{prop}
\begin{proof}
    It suffices to prove that for all $e \in E(G)$, given $J^{e-}$, the conditional probability that $e$ is green is at least $1-\mu^2/3$. So, let ${\mathcal F}_i$ be the family of 
    graphs $G$ for which $J^{e-}$ is the given partially coloured graph and for which $e$ has length exactly $i$.

    We consider an auxiliary bipartite graph between ${\cal F}_{i+1}$ and ${\cal F}_i$ where an edge joins an element of $\cF_{i+1}$ to an element of $\cF_i$ if the former can be obtained from the latter by suppressing a vertex on a green path with length at least 6 and adding it at the end of the path corresponding to $e$, thus increasing the length of $e$ by 1. 

    As noted earlier, since $n_2 \geq \left(\frac{1}{\delta} - 1\right)m^{\neq 2}$ where $\delta = \mu^3$, there are at least $\left(1 - \frac{4\delta}{1-\delta}\right)n_2 = \left(\frac{1-5\mu^3}{1-\mu^3}\right)n_2 \geq (1-5\mu^3)m$ degree two vertices on paths corresponding to (green) edges of length at least six, and so $\delta(\cF_i) \geq (1-5\mu^3)m - i$. On the other hand, there are at most $m$ edges which we could choose to subdivide so $\Delta(\cF_{i+1}) \leq m$.

    So, for $i < \mu^2 m (\frac{1}{12} - 5\mu)$ (note that this is positive as long as we assume $\mu < \frac{1}{60}$), we have $\frac{|{\cal F}_{i+1}|}{|{\cal F}_i|} > 1-\frac{\mu^2}{12}$. It follows that $\frac{|\cF_2|}{\sum_{i=2}^{k} |\cF_i|} < \frac{\mu^2}{6}$ and so $\p{\cF_1 \cup \cF_2} < \frac{\mu^2}{3}$.
\end{proof}

The following allows us to show the giant components of  $G$  are contained in the union of the large components of $J$.

\begin{prop}
    \label{sumoflengthsprop}
    Fix $\delta \in (0,1/100)$. If $m^{\neq 2} \leq \delta m$ then  for any function $f(x)$ which goes to infinity as $x \to \infty$, whp the number of 
    vertices on the paths of $G$ corresponding to any family  of at most  $\frac{m^{\neq 2}}{f(m^{\neq 2})^{2}}$ edges of $J$ is at most $\frac{m}{f(m^{\neq 2})}$. 
\end{prop}

\begin{proof}
We know $m^{\neq 2}$ is $\omega(1)$ and hence so is $g$. 
Thus for large 
$n$, by Proposition~\ref{redandyellowpathsprop} we have $\frac{N}{g-1} \le (\frac{5}{4}+o(1))\frac{m}{m^{\neq 2}}$. 
Applying Proposition \ref{greenpathsprop}, the expected number of vertices on paths of length exceeding $\frac{f(m^{\neq 2})m}{2m^{\neq 2}}$ is $o(\frac{m}{f(m^{\neq 2})})$ and hence whp the sum of the lengths of these paths is $o(\frac{m}{f(m^{\neq 2})})$. On the other hand, for a set of $\frac{m^{\neq 2}}{f(m^{\neq 2})^2}$ edges in $J$, the number of vertices on paths of length less than $\frac{f(m^{\neq 2})m}{2m^{\neq 2}}$ is at most $\left(\frac{m^{\neq 2}}{f(m^{\neq 2})^2}\right)\left(\frac{f(m^{\neq 2})m}{2m^{\neq 2}}\right) = \frac{m}{2f(m^{\neq 2})}$. Putting these together gives the claim. 
\end{proof}

Lastly, we give an upper bound on the overall number of short edges and the density of subgraphs with many edges. 
\begin{prop}
    \label{sumoflengthsprop2}
    Fix $\delta \in (0,1/100)$. If $m^{\neq 2} \leq \delta m$ then 
    whp 
    \begin{enumerate}[(i)]
    \item the number of edges of $J$  which have length less than $\frac{m}{m^{\neq 2}   \log \log \log m}$ is less than $\frac{2m^{\neq 2}}{\sqrt{\log \log \log m}}$, and \item for every constant $A$,  every set of 
    more than $\frac{m^{\neq 2}}{A}$ edges of $J$ contains more than  
    $\frac{m}{2A   \log \log \log m}$ vertices.  
    \end{enumerate}
\end{prop}
\begin{proof}
(ii) follows immediately from (i) so we need only prove the latter. 
 It is vacuous  unless  $m \ge m^{\neq 2} \log \log \log m$ 
 and hence $N=n_2-2g-y>n_2-5m^{\neq 2}>\frac{m}{3}$ so we assume this is the case. Applying 
 Proposition \ref{redandyellowpathsprop}, we have that whp, the
 number of red and yellow  edges is less than $\frac{m^{\neq 2}}{\sqrt{\log \log \log m}}$. It remains to bound the number of 
 short green edges.
 
 Using the balls and bins perspective described above, the probability a specific draw is a delimiter is $\frac{g-1}{N}  \le \frac{m^{\neq 2}}{m/3}$. So, the expected number of delimiters chosen in  the first $\frac{m}{m^{\neq 2} \log \log  \log m}$  draws  is less than  $\frac{3}{\log \log  \log m}$. Hence the probability that
 the first (and therefore any) green edge has  length less than $\frac{m}{m^{\neq 2} \log \log  \log m}$ is less than  $(\frac{3}{\log \log  \log m})$ . So, the expected  number of green edges of length less than $\frac{m}{m^{\neq 2} \log \log  \log m}$  is less than  $\frac{3m^{\neq 2}}{\log \log  \log m}$ and applying Markov's inequality whp there 
 are at most $\frac{m^{\neq 2}}{\sqrt{\log \log \log m}}$ such short green edges. 
\end{proof}

\section{Degree Sequences without a Center}\label{sec:nocenter}

In this section, we prove Theorem \ref{thm:mainnew5}. 
Thus, throughout the section, we continue to work in the setting of a sequence of degree sequences such that for some $\rho<\rho_0$ and 
$\mu>0$ sufficiently small in terms of $\rho$, for all large enough $\ell$, $\cD_{\ell}$ 
does not have a $\mu$-center and $R_{\cD_{\ell}} \ge \rho m^{\ne 2}=\omega(1)$. 
We stress that all the results we prove hold for all $\mu$ 
sufficiently small which will allow us to choose  $\mu$
to be smaller than another constant near the end of the section.
(Recall that we say a degree sequence $\cD=(d_1,\ldots,d_n)$ does not have a $\mu$-center if either $\sum_{n-d_n+1}^n d_i < \mu^2 m^{\neq 2}$ or $m^{\neq 2} < \mu^3 m$. We exclusively consider such degree sequences without explicitly restating this throughout the section.)

We note that Proposition \ref{fewcycles} implies that whp every cycle component has $O(\frac{m \log m^{\neq 2}}{m^{\neq 2}})$ vertices and hence 
its  diameter is  $O(\frac{m \log m^{\neq 2}}{m^{\neq 2}})$ and its mixing
time is $O\left(\left(\frac{m \log m^{\neq 2}}{m^{\neq 2}}\right)^2\right)$. So we need only consider the non-cycle  components.

As discussed previously, our  approach is to bound the conductance of sets of various orders. We bound  the conductance of  connected sets of vertices of  $G(\cD)$   using a number of different approaches depending on  their intersection with $J$. 
In order to do so, we need to extend the definition of conductance to subgraphs. 
The {\em conductance of a subgraph} $F$ of a graph is 
$\frac{d_{G}(V(F))-2|E(F)|}{d_{G}(V(F))}$. So the conductance of a set $S$ is simply the conductance of the subgraph it induces. 

Letting $L_{\max}$ be the maximum of 2 and the length of the longest green path,
if $S$ is  completely contained in a  green, yellow, or red path then its conductance is at least $\frac{1}{L_{\max}}$. Otherwise, we bound the conductance of $S$ by bounding in the coloured reduction the conductance of the connected subgraph $F_S$ of $J$ whose vertices are those vertices of $J$ contained in $S$ and whose (coloured) edges are those for which the entirety of the corresponding path is contained in $S$.
In this case, we note that there are at least $d_G(V(F_S))-2|E(F_S)|$ paths containing an edge between $S$ and $V-S$, 
 so the conductance of $S$ is at least
 $\frac{d_G(V(F_S))}{d(S)}\cond(F_S)>\frac{\cond(F_S)}{2L_{\max}}$. 
Furthermore, we can improve this bound by obtaining 
 bounds on the number of vertices on specified sets of $r$ edges of $J$.

We first bound the number and conductance of connected subgraphs $F$ of $J$  for various choices  of $d_G(F)$ 
and $|E(F)|$ and then use this to bound the conductance of connected subsets in the union of the giant components 
of $G$.

For subgraphs $F$ with $d(F)=O(\log m^{\neq 2})$, we will apply the following observation.

\begin{obs}
\label{easyobs}
 For any $F \subseteq J$ which is not a component of $J$,     $\cond(F)> \frac{1}{d(V(F))}$. 
\end{obs} 

 For a constant $B=B(\rho)$ specified below, we bound the conductance of $F$ with $d(V(F))$ between $B \log m^{\neq 2}$ and $\frac{m^{\neq 2}}{B}$   using coloured switchings in  $J$ 
 as discussed above. We note that we will be able to choose $B$ independently of $\mu$ provided we take
 $\mu <  \frac{\rho }{4}$. 

  Recall that the {\em core} of $G$ is the maximal subgraph of $G$ of minimum degree 2. The graph $G$ is obtained from its core by adding {\it decorations}, which are trees that are disjoint from each other and the core, each of which is either a component of $G$ or is attached to  a single vertex of the core by a single edge. If the latter, call the unique vertex of the decoration which is adjacent to the core its {\em root}.

 Recall also that the \emph{kernel} $K=K(G)$ of $G$ is the homeomorphic reduction of its core. In 
 We shall see that whp, the kernel is connected (see also \cite[Theorem 1.3]{crudele}) and hence the giant component is obtained from the kernel  by replacing its edges by paths all of whose internal vertices have degree two, with the same endpoints as the edge; and then adding   decorations hanging off the 
vertices of these paths (including their endpoints in the kernel).  We call the union of such a path and the decorations hanging off its interior 
vertices a {\it decorated path}. We will use the techniques that we applied to smaller sets to obtain bounds on the probability that a decoration hanging off a vertex of the kernel or a decorated path has a large  number of edges. This allows us to translate lower bounds on the number of edges of the kernel between two large subsets of its vertices into bounds on the conductance of large connected sets in the giant component of  $G$.

The restriction that either $\sum_{n-d_n+1}^n d_i \leq \mu^2m^{\neq 2} < \frac{\rho m^{\neq 2}}{16}$ or $m^{\neq 2}< \mu^3 m < \frac{\rho^3}{64}m$ is relevant in handling both large and small sets as it allows us to bound the difference between probabilities in the unconditioned configuration model and in our model. 

For the next lemma, 
we recall that $G$ is generated via a uniformly random matching $M\in_u\cM(\cD)$, and that we write $M^*$ for the corresponding matching of the halfedges incident to vertices of degree not equal to two.  
Let $\cH = \bigcup_{i \in [n]:d_i \ne 2} \{i1,\ldots,id_i\}$ be the set of all halfedges in $M^*$, and for $h \in \cH$, let $v(h)$ denote the vertex that $h$ is incident to.
 
\begin{lem}
\label{nocenterhelpslem}
Let $F^-$ be a coloured graph 
such that $V(F^-) \subseteq \{1, \dots, n\}$ and $\sum_{i \in F^-} d_i \leq \frac{\rho m^{\neq 2}}{4}$. Conditioning on $F^-$ being a subgraph of $J$ and on the subset $M_{F^-}$ of $M^*$ 
corresponding  to its edges, the following holds with high probability over $J$. 
  Fix $v \in V(F^-)$ and a halfedge $h$ incident to $v$ which is not in $M_{F^-}$, and let $W$ be the (random) vertex incident to the halfedge $h$ is matched to in $M$. Then 
\[
\E{\left.\min\Big(d(W)-2, \frac{4}{\rho^2}\Big)\right| J, W \in V(J) \setminus V(F^-)} \geq \frac{\rho}{4}\ .
\]
\end{lem}

\begin{proof}
Throughout the proof, we condition on $J$ and on the event that $F^-$ is a subgraph of $J$ and that the corresponding halfedge matching $M_{F^-}$ is a submatching of $M^*$. For ease of reading, we omit writing this conditioning from the expected values that we bound below. 
We choose $\rho_0$ sufficiently small  that 
$(1-\frac{\rho}{4})^{-2} < 1 + \frac{2\rho}{3}$ for $\rho \in (0,\rho_0)$.
For each halfedge $h'$ which is  incident to a vertex  $u \in V(J)\setminus V(F^{-})$, 
we consider switchings between the families ${\mathcal F}_{yes}$ and $\cF_{no}$, where $\cF_{yes}$ (respectively, $\cF_{no}$) is the collection of matchings $M \in \cM$ that correspond to graphs $G$ such that $J(G)$ contains $F^-$ as a subgraph and such that $hh' \in M^*$ (respectively, $hh' \notin M^*$).

{\bf Case 1: $m^{\neq 2} \le \mu^3 m$.} 

To obtain an element of ${\mathcal F}_{no}$  by a coloured switching from an element $M$ of ${\mathcal F}_{yes}$, 
using one of the two orderings of the edge $hh'$ we must switch  with some ordered matched pair of halfedges $bb' \in M^*$.
Furthermore, if $hh'$ is green, we can switch with any $bb'$ which is also green and such that
$v(b),v(b') \not \in V(F^-)$. 
By Proposition~\ref{specificedgeredoryellowprop}, conditioned on $F^-$ the probability that both $hh'$ and $bb'$ are green is at least $1-\mu^2$. Thus, the number of edges out of elements of $\cF_{yes}$ in our auxiliary bipartite graph is at least $(4-4\mu^2)(1-\frac{\rho}{4})m^{\neq 2}|{\mathcal F}_{yes}|$ and at most $4m^{\neq 2}|{\mathcal F}_{yes}|$.

To obtain an element of ${\mathcal F}_{yes}$ by switching from one in ${\mathcal F}_{no}$,
we must switch one of the two orderings of the edge containing $h'$ and one of the two orderings of the 
edge containing $h$. Furthermore, at most two of these pairs of orderings will work, and exactly two unless one of the edges we are switching is yellow or red. Again by Proposition~\ref{specificedgeredoryellowprop} and symmetry, 
we see that  for any $h'$, the number of switchings from 
${\mathcal F}_{no}$ is between  $(2-2\mu^2)|{\mathcal F}_{no}|$  and  $2|{\mathcal F}_{no}|$.

For every $h$ and $h'$, the conditional probability that $h$ is joined to $h'$ is equal to $\frac{|\cF_{yes}|}{|\cF_{yes}| + |\cF_{no}|}$. Applying the switching inequality \eqref{switchingineq}, we see that $$|\cF_{no}| \leq \frac{4m^{\neq 2}}{2-2\mu^2}|\cF_{yes}| = \frac{m^{\neq 2}}{1-\mu^2}|\cF_{yes}|$$ 
and 
$$|\cF_{no}| \geq \frac{(4-4\mu^2)(1-\frac{\rho}{4})m^{\neq 2}}{2}|\cF_{yes}| = (2-2\mu^2)(1- \frac{\rho}{4})m^{\neq 2}|\cF_{yes}|\ .$$
The above conditional probability is thus at least $\frac{1-\mu^2}{2m^{\neq 2} + 1-\mu^2} = \frac{1-\mu^2}{(1+o(1))2m^{\neq 2}}$ and at most
$\frac{(1-\mu^2)^{-1}(1-\frac{\rho}{4})^{-1}}{2m^{\neq 2}} \leq \frac{(1 - \frac{\rho}{4})^{-2}}{2m^{\neq 2}} \leq 
\frac{1+\frac{2\rho}{3}}{2m^{\neq 2}}$. 

 To compute the expression in the claim, recall we condition on $h$ being matched to a halfedge at vertex $W$ outside of $F^-$,  so we define a function $f$ on $\cH$ where
$$f(h') = \begin{cases} 1\ \ &\text{ if }v(h') \in V(F^-)\\
0\ \ &\text{ else.}
\end{cases}$$

Recalling that $d(v) > 0$ for all $v \in V$, we partition $\cH$ into $\cH_1 := \{h' \in \cH : d(v(h')) =1\}$ and let $\cH_2 = \{h' \in \cH : d(v(h')) \geq 3\}$. For $h' \in \cH_1$, we have $\min\{d(v(h')) - 2, \frac{4}{\rho^2}\} = -1$. 
For $h' \in \cH_1$ we use the upper bound on the probability above, and for $h' \in \cH_2$ we use the lower bound, to obtain
\begin{align*}&\E{\min\left\{2m^{\neq 2}(d(W)-2), 2m^{\neq 2}\frac{4}{\rho^2}\right\}} \geq \\
& \ \ \geq \sum_{h' \in \cH_2} (1-f(h'))(1-o(1))(1-\mu^2)\min\left\{d(v(h'))-2, \frac{4}{\rho^2}\right\} + \sum_{h' \in \cH_1}\left(1+\frac{2\rho}{3}\right)(-1)\\
& \ \ \geq (1-o(1))\sum_{h' \in \cH}\left((1-f(h'))(1-\mu^2)\min\left\{d(v(h'))-2, \frac{4}{\rho^2}\right\} - \left(\frac{2\rho}{3} + \mu^2\right)\right)\\
& \ \ \geq (1-o(1))\sum_{h' \in \cH}\left((1 - f(h'))\left(1-\frac{\rho^2}{16}\right)\min\left\{d(v(h')) - 2, \frac{4}{\rho^2}\right\} - \left(\frac{2\rho}{3} + \frac{\rho^2}{16}\right)\right), 
\end{align*}
where in the last line we use the assumption $\mu \leq \frac{\rho}{4}$.

We may bound the final expression from below by
\[
\min_{f^*} \left[(1-o(1))\sum_{h' \in \cH}\left((1 - f^*(h'))\left(1-\frac{\rho^2}{16}\right)\min\left\{d(v(h')) - 2, \frac{4}{\rho^2}\right\} - \left(\frac{2\rho}{3} + \frac{\rho^2}{16}\right)\right)\right], 
\]
where the minimum is taken over all non-negative 
functions $f^*$ on $\cH$ which have maximum 1 and sum to at most $\frac{\rho m^{\neq 2}}{2}$.
Recalling that $j_{\D}$ is defined as the minimum $j \in [n]$ such that $\sum_{i=1}^j d_i(d_i - 2) > 0$, the above sum is minimized when 
$f^*(h') \in \{0,1\}$ for all $h' \in \cH$, and every $h'$ for which $f^*(h') = 1$ satisfies $d(v(h')) \geq d(v_{j_\D})$.
Now, by Theorem~\ref{whenhasgiantthm}, provided that $n$ is sufficiently large there are at least  $\frac{ \rho m^{\neq 2}}{2}$ halfedges incident to vertices $v(h')$ for which $f^*(h')=0$ and $d(v(h'))\ge d_{j_{\mathcal D}}$. 

If $d_{j_{\D}}>\frac{4}{\rho^2}+2$ then we are done.
Indeed, the summation over the halfedges for which $d(v(h')) \geq d_{j_\D}$ contributes at least 
$$\frac{\rho m^{\neq 2}}{2} \left(\left(1 - \frac{\rho^2}{16}\right)\left(\frac{4}{\rho^2}\right)\right) - m^{\neq 2}\left(\frac{2\rho}{3} + \frac{\rho^2}{16}\right) > \frac{2m^{\neq 2}}{\rho}\left(1-\frac{\rho^2}{16}\right) - \frac{3\rho m^{\neq 2}}{4}$$
which is at least $\frac{\rho m^{\neq 2}}{4}$ since $\rho < 1$.

Otherwise, we split our sum into two parts, the first over $h'$ such that $d(v(h')) \leq d_{j_\D}$ and the second over the remaining $h'$. For the first part of the sum, we count $d(v(h'))(d(v(h'))-2)$ for all vertices and by definition of $j_\D$, this sum is positive. For the second part of the sum, we have a lower bound of at least $\frac{\rho m^{\neq 2}}{2}$, so again we are done.

{\bf Case 2: } $m^{\neq 2} > \mu^3 m$, which implies $\sum_{i=n-d_n+1}^{n} d_i \le \frac{\rho m^{\neq 2}}{4}$ 

As in Case 1, consider how we obtain an element of ${\mathcal F}_{no}$ from an element of ${\mathcal F}_{yes}$ 
by switching an ordering of $hh'$ with some ordering of an edge $bb'$. 
If $d(v(h')) = 1$, we can switch $hh'$ with any $bb'$ for which $v(b), v(b') \not \in N(v) \cup V(F^-)$. By symmetry, we may do the same with $h'h$ and $b'b$. 
By our assumptions that $m^{\neq 2} > \mu^3 m$ and $\sum_{i \in F^-} d_i \leq \frac{\rho m^{\neq 2}}{4}$, the number of such switchings is at least $4(1-\frac{\rho}{4})m^{\neq 2}$ if $h'$ is incident to a vertex of degree 1, and is at most $4m^{\neq 2}$.

Given an element of $\cF_{no}$, suppose $h$ is matched with a halfedge $b \neq h'$ and $h'$ is matched with a halfedge $b'$. To obtain an element of $\cF_{yes}$, we must swap an ordering of $bh$ with an ordering of $b'h'$. Furthermore, at most two of these possibilities work and exactly two if $b', h'$ are not incident to $v(b)$ or $v$ or a neighbor of $v(b)$ or $v$.
Let $q(h')$ be the probability that at least one of $h', b'$ is incident to a vertex in $\{v, v(b)\} \cup N(v) \cup N(v(b))$. In other words, if we let $Q^{h'}$ be the set of all matchings in $\cF_{no}$ for which the above holds, then $q(h') = \frac{|Q^{h'}|}{|\cF_{yes}| + |\cF_{no}|}$.
We then have
$$|\cF_{no}| \leq \frac{4m^{\neq 2}}{2(1-q(h'))}|\cF_{yes}| = \frac{2m^{\neq 2}}{1-q(h')}|\cF_{yes}|\ , $$
$$|\cF_{no}| \geq 2(1-\frac{\rho}{4})m^{\neq 2}|\cF_{yes}|, $$
the second bound holding when $d(v(h')) = 1$.
We obtain that $\frac{|\cF_{yes}|}{|\cF_{yes}| + |\cF_{no}|}$ is at least 
$\frac{1-q(h')}{1-q(h') + 2m^{\neq 2}} \geq \frac{1-q(h')}{(1+o(1))2m^{\neq 2}}$, and is at most $\frac{1+\frac{\rho}{3}}{2m^{\neq 2}}$ if $d(v(h')) = 1$.

By the assumptions of Case 2, the size of the second neighborhood of any vertex is at most $\frac{\rho m^{\neq 2}}{4}$. This means that every matching in $\cM(\cD)$ is contained in at most $\frac{\rho m^{\neq 2}}{4}$ of the sets $Q^{h'}$.
Thus, 
the sum of $q(h')$ over all $h'$ incident to vertices of $V(J)\setminus V(F^-)$ is also at most $\frac{\rho m^{\neq 2}}{4}$. 

Now, define $f$ on $\cH$ by 
$$f(h') = \begin{cases} 
1&\ \ \text{if } v(h') \in V(F^-) \\
q(h')&\ \ \text{else. }
\end{cases}$$
The above remarks imply that $\sum_{h' \in \cH} f(h') \le \frac{\rho m^{\ne 2}}{2}$.
Conditioned on the vertex $W$ 
matched to $h$ being outside $F^-$, we then have that
\begin{align*} &\E{\min\left\{2m^{\neq 2}(d(W)-2), 2m^{\neq 2}\frac{4}{\rho^2}\right\}} \geq\\
&\ \ \ \geq (1-o(1))\sum_{h' \in \cH} \left((1-f(h'))\min\left\{d(v(h')) - 2, \frac{4}{\rho^2}\right\} - \frac{\rho}{3} - 1\right)
\end{align*}
The rest of the argument proceeds just as above, by bounding the right-hand side from below by minimizing over functions $f^*$ on $\cH$ which have maximum one and sum to at most $\frac{\rho m^{\neq 2}}{2}$.
\end{proof}

\subsection{Bounding the conductance of small subgraphs of the coloured reduction}
In this subsection, we use Lemma~\ref{nocenterhelpslem} to bound the conductance of connected subgraphs $F$ of the coloured reduction for which $C\log m^{\neq 2} \le d(V(F)) \le \frac{m^{\neq 2}}{C}$ where $C$ is a constant depending only on $\rho$. We start by bounding the {\em excess} of such subgraphs $F$,  which we define as $\ex(F) = d(V(F)) - 2|V(F)| + 2$. 
To do this, we count the expected number of submatchings of $M^*$ corresponding to rooted spanning subtrees  of $F$ using an exploration process which we now describe.

For positive integers $d,\ones$, and  a submatching of $M^*$ 
corresponding to a  spanning tree $T$ of $J$ rooted at $v$ with $\ones$ edges (and so $\ones+1$ vertices) and such that $d_J(T) :=\sum_{w \in V(T)} \deg_J(w)=d$ we may associate a binary string of length $d-\ones$ with exactly $\ones$ ones. 
We expose the vertices of $T$ starting from $v_1=v$ in a breadth-first manner, exposing  the children of $v_i$  (and hence their degree) and the paired halfedges of $M$ corresponding to the edges to these children when we  explore
from $v_i$.  We expose the  children of each vertex on a level  in increasing order of label. For each vertex $w$ of the tree, we define  $d^{(w)}$ to be the degree of $w$ in $J$ outside of the subtree of $T$ we have already exposed  when exploring from $w$ (which is $\deg_J(w)$ if $w=v$ and is $\deg_J(w)-1$ otherwise). We consider the  $ d^{(w)}$ unused halfedges from $v_i$ in order and use the next $d^{(w)}$ bits of the string to record which of these  halfedges  are in edges to new vertices of $T$. Ones represent edges to new vertices of $T$, and we add the corresponding edges to $T$ and corresponding edges to our submatching 
of $M^*$. Zeros represent either edges to vertices already added to $T$ or edges to vertices not in $T$. The tree $T$ and corresponding 
submatching may be recovered from $M^*$ and the resulting binary string since the breadth-first exploration order is deterministic.

Writing $\cB$ for the set of a binary strings of length $d-\ones$ with exactly $\ones$ ones, the preceding paragraph describes how every submatching of $M^*$ corresponding to a  tree $T$ containing $v$ with $\ones$ edges and $d(V(T))=d$ uniquely corresponds to a string in $\cB$. However, not every binary string in $\cB$  is associated to such a matching.
In particular, for a string in $\cB$ to be associated to such a matching, all its ``one'' entries must correspond to edges to vertices of $J$ that have not already been added to $T$, and the degrees in $J$ of all the vertices that are added to the tree must sum to exactly $d$. Say that a bit string $b$ is {\em valid} for $J$ and $v$ if both of these properties hold. Then it follows that the number of such submatchings of $M^*$ is 
\begin{equation}\label{treenumber}
\sum_{b \in \cB} \I{b\text{ is valid}} \le 
{d -\ones \choose d-2\ones} \le \left(\frac{e(d-\ones)}{d-2\ones}\right)^{d-2\ones}.
\end{equation}

We will apply this idea using Lemma~\ref{nocenterhelpslem} to help bound  probabilities, and obtain the following.

\begin{lem} 
\label{countsmallexcess}
There exist $\epsilon,\gamma,B>0$, and $a_0$    such that for any $a$ with $a_0 \le a \le \frac{m}{B}$,
the expected number of submatchings of $M$  
corresponding to trees with  $a$ edges   and $d_J(T) \leq (2+\epsilon)a$ is at most $\frac{m^{\neq 2}2^{-\gamma a}}{a}$ and hence the expected number of vertices in such trees is at most $m^{\neq 2}2^{-\gamma a}$. Furthermore, for $a \le \log \log m$,  the expected number of ordered sets of $\lceil \log m \rceil$  vertices in disjoint such trees is at most  $(m^{\neq 2} 2^{-\gamma a})^{\lceil \log m \rceil}$.
\end{lem}

\begin{proof}

We choose $\epsilon=\rho^3$ and  $B$  so that $(2+\epsilon)\frac{m}{B} \le \frac{\rho m^{\neq 2}}{4}$. 
For a given string $b \in \cB$ and a given vertex $v \in V(G)$, we condition on $v$ being a vertex in $J$, and try to explore a subtree $T$ of $J$ corresponding to $b$. By Lemma~\ref{nocenterhelpslem}, at each step that we explore a new vertex $W$ of $T$, conditional on the coloured graph $F^-$ corresponding to the subtree of $T$ that has already been explored, we have 
\[
\E{\left.\min\Big(d(W)-2, \frac{4}{\rho^2}\Big)\right|W \in V(J) \setminus V(F^-)} \geq \frac{\rho}{4}, 
\]
provided that $\sum_{i \in F^-} d_i \le \frac{\rho m^{\ne 2}}{4}$.
It follows that for $d \le (2+\epsilon )a \le  \frac{\rho m^{\ne 2}}{4}$ we can bound the total degree revealed during the exploration from below by 
\[
\min\left\{\frac{\rho m^{\ne 2}}{4},\ \sum_{i=1}^{\ones+1} D_i\right\},
\]
where $(D_i,i \ge 1)$ are random variables with the property that $\E{D_i|T_{i-1}} \ge 2+\frac{\rho}{4}$ and such that $0 \le D_i \le 2+\frac{4}{\rho^2}$. (Here $T_i$ represents the subtree of $T$ revealed by the exploration of the first $i$ vertices starting from $v$.)

   By a union bound over the starting vertex, for  any $d \le (2+\epsilon)a$, the expected number of subtrees in $J$ with $\ones$ edges and $d_J(T) = d$ is at most

\begin{align*}
& m^{\ne 2} \left(\frac{e(d-\ones)}{d-2\ones}\right)^{d-2\ones}\p{\sum_{i=1}^{\ones+1}D_i =d} \\
& \le m^{\ne 2} \left(\frac{ed}{d-2\ones}\right)^{d-2\ones}\p{\sum_{i=1}^{\ones+1}(D_i-\e D_i) \le d-\left(2+\frac{\rho}{4}\right)(\ones+1)}
\end{align*}
  
  Applying a martingale concentration bound  such as \cite[Theorem 3.7]{MR1678578}, provided that $\rho_0$ is sufficiently small, this 
  expectation is less than 
  \[
m^{\ne 2} e^{-\rho^2 \ones/64 }. 
\]
  Summing over $ d\le (2+\rho^3)a$ provides the first desired result. The second follows 
  by analyzing a choice of exploration for the $\lceil \log m \rceil$ trees in the same way.  

\end{proof}

\begin{cor} 
\label{nosmallexcess}
There exist $\epsilon = \epsilon(\rho) > 0$ and $C = C(\rho) > 0$ such that the probability $J(G)$ contains a tree $T$ with $C \log m^{\neq 2} \leq d_J(T) \leq \frac{m^{\neq 2}}{C}$ and $d_J(T) \leq (2+\epsilon)|V(T)|$ is $o(1)$.
\end{cor}
\begin{proof}
We take  $C>0$ sufficiently large,  sum the expectation of the number of bad trees over  $\frac{C}{6} \log m^{\ne 2} \le a \le \frac{m^{\ne 2}}{2C}$, and apply Markov's  inequality, 
\end{proof}

The preceding corollary implies that if $C \log m^{\neq 2} \leq d_J(T) \leq \frac{m^{\neq 2}}{C}$, then whp $\ex(V(T)) > \frac{\epsilon}{2+\epsilon} d_J(T)$. 
We use this to prove the following. 

\begin{lem}\label{lem:sparse-conductance} 
    There exists $C = C(\rho) >0$ such that with high probability every subgraph $F$ of $J(G)$ with $C\log m^{\ne 2} \le d_J(F) \le m^{\ne 2}/C$ satisfies $\cond(F) \ge 1/C$.
\end{lem}
\begin{proof}
Let $\eps'=\frac{\eps}{2+\eps}$, and for a tree $T$ with $V(T)\subset [n]$ write $d(T):=\sum_{i \in V(T)} d_i=d$.   
We prove the lemma by showing that for $C>0$ sufficiently large, the expected number of subtrees $T$ of $J(G)$ satisfying 
\begin{enumerate}
    \item $C \log m^{\ne 2} \le d(T) \le m^{\ne 2}/C$, 
    \item $\ex(V(T)) \ge \eps' d(T)$, and 
    \item $\cond(V(T)) < 1/C$
\end{enumerate}
tends to $0$ as $n \to \infty$. As observed above, we may restrict our attention to subtrees $T$ with $\ex(V(T)) \ge \eps'  d_J(T)$ as a result of \Cref{nosmallexcess}.

 Recall that for $S \subset V$, we let $\out(S)$ be the number of edges between $S$ and $V-S$. Note that then $$\cond(V(T)) \leq \frac{\epsilon'}{C} \Rightarrow \out(V(T)) \leq \frac{\epsilon' d_J(T)}{C} \Rightarrow \out(V(T)) \leq \frac{\ex(V(T))}{C}.$$
Thus, writing $N$ for the number of subtrees $T$ of $J$ satisfying (1), (2), and $\out(V(T)) \leq \frac{\ex(V(T))}{C}$, it suffices to show that $\e[N]$ tends to $0$ as $n \to \infty$. 

Fix $d$, $\ex$, and a tree $T$ whose vertex set $V(T)$ is a subset of $\{i \in [n]: d_i \ne 2\}$, with $d(T)=d$ and with $\ex(V(T))=\ex$.

Let $\cF_i$ be the set of matchings $M \in \cM$ corresponding to graphs $G$ such that $J(G)$ contains $T$ as a subgraph, and there are exactly $2i$ edges in $J(G)$ from $V(T)$ to $V(J) \setminus V(T)$.
Given an element of $\cF_i$, there are at most $(2i)(2i-1) \leq (\ex)^2$ ordered pairs of boundary edges that can be switched to obtain an element of $\cF_{i-1}$. 

In the other direction, fix an element $M$ of $\cF_{i-1}$ corresponding to a graph $J$, which necessarily contains $T$. For any edges $e,e'$ with $e \in E(J[V(T)])\setminus E(T)$ and $e'$ an edge of $J-V(T)$, there are two choices for a switch to $\cF_i$ unless one of $e,e'$ is not green. If $m^{\ne 2} < \mu^3 m$ then by Proposition~\ref{specificedgeredoryellowprop}, conditional on $J^-$, the probability that both are green is at least $1-\mu^2$. 

Since there are $\frac{\ex - (2i-2)}{2}$ edges in $E(J[V(T)])\setminus E(T)$,  and $m^{\neq 2}-d > m^{\neq 2}-\frac{m^{\neq 2}}{C}$ edges in $E(J - V(T))$, it follows that if $m^{\neq 2}<\mu^3 m$ then the number of switches from $\cF_{i-1}$ to $\cF_i$ is at least
    $$\frac{\ex - 2i+2}{2}\left(1-\frac{1}{C}\right)m^{\neq 2}(1-\mu^2)|\cF_{i-1}|\ .$$

If $m^{\neq 2} \geq \mu^3 m$, we count the number of pairs of edges $e, e'$ that cannot be switched. Given $e \in E(J[V(T)])\setminus E(T)$, the edges $e'$ which cannot be switched with $e$ are those incident to vertices in the second neighborhood of either endpoint of $e$. By our assumption that $\cD$ does not have a center, there are at most $\sum_{n-d_n+1}^n d_i \leq \frac{\rho m^{\neq 2}}{2}$ such edges $e'$. In this case, the number of switches from $\cF_{i-1}$ to $\cF_i$ is at least
$$\frac{\ex - 2i+2}{2}\left(1-\frac{1}{C}\right)m^{\neq 2} - (\ex-2i+2)\frac{\rho m^{\neq 2}}{4} = (\ex-2i+2)m^{\neq 2}\left(\frac12\left(1-\frac{1}{2C}\right) - \frac{\rho}{4}\right)\ .$$

Letting $C' = (2\min\left\{\frac12\left(1 - \frac{1}{C}\right)(1-\mu^2), \frac{1}{2}\left(1 - \frac{1}{2C}\right) - \frac{\rho}{4}\right\})^{-1}$, then it follows from the above case analysis that
    $$\frac{|\cF_{i-1}|}{|\cF_i|} \leq \frac{C'}{2}\cdot\frac{\ex^2}{(\ex-2i+2)m^{\neq 2}}$$
Applying this to $i \leq \frac{\ex}{4}+1$ gives
    $$\frac{|\cF_{i-1}|}{|\cF_i|} \leq \frac{C'}{2} \cdot \frac{\ex^2}{\frac{1}{2}(\ex) m^{\neq 2}} = C'\left(\frac{\ex}{m^{\neq 2}}\right)$$
and in particular, for $k \leq \frac{\ex}{4}+1$,
    $$\sum_{i \leq k} |\cF_i| \leq 2|\cF_{2k}| \left(C' \frac{\ex}{m^{\neq 2}}\right)^k\ .$$
Taking $k = \frac{\ex}{C}$, it then follows that
    $$\p{\out(V(T)) \leq \frac{\ex}{C}\ |\ T \subseteq J} = \frac{\sum_{i\leq \ex/C} |\cF_i|}{\sum_{i=0}^{m^{\neq 2}} |\cF_i|} \leq \frac{\sum_{i\leq \ex/C}|\cF_i|}{|\cF_{2\ex/C}|} \leq 2\left(C'\frac{\ex}{m^{\neq 2}}\right)^{\ex/C}$$
As $\ex < d$ and by our assumption $d < \frac{m^{\neq 2}}{C}$, the right-hand expression is at most $2 \left(\frac{C'}{C}\right)^{\ex/C}$. 
It follows that 
\begin{align*}
\E N& \le \sum_{d = C\log m^{\neq 2}}^{m^{\neq 2}/C}\sum_{\ex = \epsilon'd}^{d-1} \sum_T \p{(T \subset J) \wedge (\out(V(T)) \leq \frac{\ex}{C})}\, ,
\end{align*} 
where the inner sum is over trees $T$ with $d(T)=d$ and $\ex(V(T)) = \ex$. 
We may rewrite the inner sum as 
\begin{align*}
 & \sum_{\{T:d(T)=d,\ex(V(T))=\ex\}}\E{\I{T\subset J}\I{\out(V(T)) \le \frac{\ex}{C}}\big\}} \\
 & = \sum_{\{T:d(T)=d,\ex(V(T))=\ex\}}\p{T\subset J}\cdot 
 \p{\out(V(T)) \leq \frac{\ex}{C}\ |\ T \subseteq J} \\
 & \le 
 2 \left(\frac{C'}{C}\right)^{\ex/C}\cdot \E{\#\{T\subset J:d(T)=d,\ex(V(T))=\ex \}} \\
 & \le 
 2 \left(\frac{C'}{C}\right)^{\ex/C}\cdot
 m^{\ne 2} \left(\frac{ed}{\ex}\right)^{\ex}\, ,
\end{align*}
the last bound holding by \eqref{treenumber}.

By taking $C$ large enough, and noting that $C'$ is bounded from above as $C$ becomes large, we therefore obtain that 
\begin{align*}
    \e[N]&\le  \sum_{d = C\log m^{\neq 2}}^{m^{\neq 2}/C}\sum_{\ex = \epsilon'd}^{d-1} 2 \left(\frac{C'}{C}\right)^{\ex/C}\cdot
 m^{\ne 2} \left(\frac{ed}{\ex}\right)^{\ex} \\
    & \le (m^{\neq 2})^42^{-\ex/C}\ .
\end{align*}
Lastly, since $\ex>\epsilon' d > \epsilon'  C \log m^{\neq 2}$ and since $m^{\ne 2} \to \infty$ as $\ell \to \infty$, for $C$ large enough the final expression tends to 0 as $\ell \to \infty$.
\end{proof}

\subsection{Bounding the conductance of sets which intersect the kernel }
\label{subsection:kernel}

In order to lower-bound the conductance of sets  of vertices of $G$ which intersect the kernel, 
we consider  (i) the number of edges leaving  subsets of vertices of the kernel, (ii) the size of 
the  decorations hanging off its vertices,  and (iii) the size of the  decorated paths corresponding to its edges.

\subsubsection{Bounding the size of decorations and decorated paths}

We begin by showing that high-degree vertices must lie in the kernel $K$ of $G$, and that their degrees when restricted to the kernel remain large.

\begin{lem}
\label{highdegreeinthekernel}
    With high probability, every $v$ with $d_J(v)> \frac{m^{\neq 2}}{\log^2 m^{\neq 2}}$ is in the kernel and satisfies $d_K(v)>\frac{\rho^2\mu^6d(v)^2}{2^{12}m^{\neq 2}}$.
    \end{lem}

\begin{proof}
    There are at most $2\log^2 m^{\neq 2}$ such  high-degree vertices.   
    So it is enough to show that for a specific such vertex $v$, 
    the probability that $v$ is not in the kernel or that $d_K(v) \le \frac{\rho^2\mu^6d(v)^2}{2^{12}m^{\neq 2}}$ is  $o(\log^{-2} m^{\neq 2})$. 
    In other words, let $\cF_{i}$ be the family of matchings that correspond to graphs $G$ such that in $J(G)$, the 
    number of edges incident to $v$ which lie on paths  corresponding to edges of the kernel is at most $i$.   We want to show $\p{\cup_{i \le \frac{\rho^2\mu^6d(v)^2}{2^{12}m^{\neq 2}}} \cF_{i} } = o(\log^{-2}m^{\neq 2})$. In order to do so, we exploit the fact that if $v$ is in the kernel
    then any edge that is incident to $v$ that is part of a cycle must also be in the kernel.

    We first define an event $Good$ such that $\p{(Good)^c \cap  \left(\bigcup_{i \le \frac{d(v)^2}{10m^{\neq 2}}} \cF_{i} \right)} \leq \p{(Good)^c} = o(\log^{-2}m^{\neq 2})$. If $m^{\neq 2} \le \mu^3 m$, let $Good$ be the event that 
    there are at least $\frac{\rho\mu^3 d(v)}{32}$ endpoints of green edges incident to $v$.
    By Proposition~\ref{specificedgeredoryellowprop}, $\p{Good} \leq \p{X < \frac{\rho\mu^3 d(v)}{32}}$ where $X \sim \mathrm{Bin}(d_J(v), 1-\frac{\mu^2}{3})$. Applying Chernoff bounds, we obtain that  
     $\p{Good} \ge 1-o(\log^{-2} m^{\neq 2})$. 

    If $m^{\neq 2} > \mu^3 m$, let $Good$ be the event that 
    at least $\frac{\rho\mu^3 d(v)}{32}$ neighbours of $v$ have degree at least three.
    We claim that again $\p{Good} \ge 1-o(\log^{-2} m^{\neq 2})$. To see this we consider the family ${\cal K}_i$ of matchings in which $v$ is incident to $i$ vertices of degree at least three. By the standing conditions of this section, there are at least $\frac{\rho m^{\neq 2}}{2}$ edges incident to vertices of degree three, and the fact that the degree sequence has no center means at most $\frac{\rho m^{\neq 2}}{4}$ of these are incident to the closed neighborhood of $v$. So, there are at least $\frac{\rho m^{\neq 2}(d(v)-i)}{4}$ switchings from ${\mathcal K}_{i+1}$ to ${\mathcal K}_{i}$. On the other hand there are at most $(i+1)m$ switches in the other direction.
    So, for $i \le \frac{\rho\mu^3 d(v)}{16}$, we have $|{\mathcal K}_i|>2|{\mathcal K}_{i+1}|$, and the claim follows. 
   
    What remains is to show that $\p{Good \cap \cup_{i \le \frac{d(v)^2}{10m^{\neq 2}}} \cF_{i} }=o(\log^{-2} m^{\neq 2})$.
    To this end, for $i \ge 3$ (which implies $v$ is in the kernel), we consider switchings from $Good \cap {\mathcal F}_i$ to $Good \cap {\mathcal F}_{i+2}$. 

    If $m^{\neq 2} \le \mu^3 m$, given a matching in $Good \cap \cF_i$, we can switch any two green edges incident to $v$ not in the kernel,
     to obtain a green loop 
    on $v$. By the definition of $Good$ in this case, there are at least $(\frac{\rho\mu^3 d(v)}{32}-i)^2$ such switchings. On the other hand, there are 
    at most $(i+2)m$ switches in the other direction.

    If $m^{\neq 2} >\mu^3 m$, given a matching in $Good \cap \cF_i$, there are at least $(\frac{\rho\mu^3 d(v)}{32}-i)^2$ pairs $(x,y)$ 
    where both $x$ any $y$ are outside the kernel,  are incident to $v$ and have degree at least three. We obtain a graph in $\cF_{i+2}$.
    which contains the triangle $vxy$ by switching on an  edge $xz$ and an edge $wy$. 
     On the other hand, to switch in the other direction, we must use an edge which forms a component of $G[N[v]]$
    so there are at most $(i+2)m$ switchings in the other direction.

    So in either case for $3 \leq i \leq  \frac{\rho^2\mu^6d(v)^2}{2^{12}m^{\neq 2}}$, we have $|Good \cap {\mathcal F}_{i+2}|>2|Good \cap {\mathcal F}_i|$.
    Now $\cF_{1}$ and $\cF_{2}$ are empty by the definition of the kernel, and a similar argument shows that the
    size of $\cF_{0}$ is less than the size of $\cF_{4}$. 
    We are done by applying \Cref{switchingineq}. 
    
\end{proof}

\begin{lem}
\label{decorationsize}
    There exists $C=C(\rho)>0$ such that with high probability, each decoration, decorated path, and unicyclic component contains at most
    $C \log m^{\neq 2}$ edges of $J$.  
\end{lem}

\begin{proof}
Each decoration, decorated path, unicyclic component of $G$, and vertex not in the kernel along with the decorations hanging off of it corresponds to a connected subgraph of $J$ 
with excess less than $2$. \Cref{nosmallexcess} thus implies that whp, no such subgraph contains between $C \log m^{\neq 2}$ and $\frac{m^{\neq 2}}{C}$ edges of $J$ for some $C = C(\rho) > 0$.

Let $v$ be a vertex  in  the core  of $J$ but not the kernel, and suppose for sake of contradiction that the total number of vertices contained in decorations at $v$ is at least $\frac{m^{\neq 2}}{C}$. Let $T_1, \dots, T_{d(v)}$ be the tree decorations at $v$, and let $v_i$ be a lowest vertex in $T_i$ such that the subtree rooted at $v_i$ has at least $\frac{m^{\neq 2}}{C}$ vertices. Then, the subtrees rooted at the children of $v_i$ must have at most $C\log m^{\neq 2}$ vertices each, by assumption. By \Cref{highdegreeinthekernel}, we have $d(v_i) \leq \frac{m^{\neq 2}}{\log^2 m^{\neq 2}}$, which implies that there are at most $\frac{Cm^{\neq 2}}{\log m^{\neq 2}}$ vertices covered by the subtrees rooted at the children of $v_i$. But this is a contradiction to the choice of $v_i$.

Reapplying \Cref{nosmallexcess} we see that whp for every  
vertex  $v$ not in the kernel but in the core, the total number of vertices in the decorations hanging off $v$ is at 
most $C \log m^{\neq 2}$.    So whp, if a decorated path in $G$ contains more
than $C\log m^{\neq 2}$ edges  of $J$, then either it or a subpath of it is 
a tree with excess $2$ containing between $C \log m^{\neq 2}$  and $2C \log m^{\neq 2}$ vertices. 
\Cref{nosmallexcess} thus implies that whp there is no such  decorated path.

A similar  two-step argument applies to the unicyclic components. For each such component $K$ with cycle $Y$, we first bound the total size 
of the  components of $K-Y$  hanging off a vertex of $Y$ and then consider the subpaths of $Y$.  

 \end{proof}

\begin{lem}
\label{decorationsize2}
    There exist $\alpha>0$ and $l_0>0$  such that whp   for all $\ell>l_0$,  
     the number of edges   in 
    the decorations, decorated paths, and unicyclic components of $J$ containing more than $\ell$ edges of $J$  is at most  $2^{-\alpha \ell}m^{\neq 2}$. 
\end{lem}

\begin{proof}
    For $\ell \ge \log \log m$, this follows immediately from Lemma \ref{countsmallexcess} and Markov's Inequality. 
    For $\ell \le \log \log m$, letting $X$ be the number of  decorations, decorated paths, and unicyclic components  of size at most $\ell$,
    Lemma \ref{countsmallexcess} implies that $E[X^{\lceil \log m \rceil}] \le  (m^{\neq 2} 2^{-\gamma \ell})^{\lceil \log m \rceil}$.
    Thus, $\p{X>2m^{\neq 2} 2^{-\gamma \ell}}<2^{-\lceil \log m \rceil}$. 
\end{proof}

\subsubsection{Lower-bounding the number of edges of the kernel}
Let $K$ denote the kernel of $G$, which has minimum degree at least 3. Let $k=|V(K)|$,  and let $d(K) \geq 3|V(K)|$ be the sum over vertices in $K$ of their degrees restricted to $K$.

We first obtain a lower bound on the fraction of edges that lie in the kernel. 
\begin{lem}
\label{d(K)lower bound} 
There exists $\gamma =\gamma(\rho)>0$ such that whp, $d(K) \ge \gamma m^{\neq 2}$. 
\end{lem}

\begin{proof} 
By the standing assumptions of the section, there exists $\eps=\eps(\rho)$ such that whp $G$ has a component containing $\epsilon n$ 
vertices.  Applying \Cref{whenhasgiantthm} we obtain 
that  there is an $\epsilon '$ (which is determined by $\epsilon$ 
and hence $\rho$) such that whp there is a component of $J$
with $\epsilon' m^{\neq 2}$ edges. 
We compute how many of these edges are contributed by the kernel versus the decorations and components other than the kernel.

To begin, we compute the number $m^*$ of edges which are 
incident to vertices of the kernel.
 Fixing $\ell$ to be large enough,
 Lemma \ref{decorationsize2} implies that whp
the collection of all unicyclic components, decorations, and decorated paths  
of size exceeding $\ell$ 
spans at most $2^{-\ell\alpha}m^{\neq 2}$ edges. 
We choose $\ell$ so that in addition to this property, we also have $2^{-\ell \alpha} \le \frac{\epsilon'}{8}$. 
Thus any component of $J$ with at least $\epsilon' m^{\neq 2}$ edges is not unicyclic.
Furthermore, the decorations and decorated paths with fewer than $\ell$ edges span at most $\ell m^*$ edges in total. 
This implies that whp  the number of edges of such a component 
 is at most $(\ell+1)m^*+ \frac{\epsilon' m^{\neq 2}}{8}$.
So whp $m^* \ge \frac{7\epsilon' m^{\neq 2}}{8(\ell+1)}$.

We now bound $E(K)$ as a function of $m^*$.
Let $High$ be the set of vertices  of $K$ which have degree at 
least $C \log m^{\neq 2}$. 

If there exists a vertex in $High$ such that the sum of the sizes of the decorations hanging off  it is at least $\frac{m^{\neq 2}}{2C}$
we can choose some  subset of these decorations of size between $\frac{m^{\neq 2}}{2C}$ and $\frac{m^{\neq 2}}{C}$.
Applying  Lemma \ref{lem:sparse-conductance}  to the union of $v$ and these decorations, we obtain that the degree of $v$ is at least $\frac{m^{\neq 2}}{2C^2}$. But now, by Lemma \ref{highdegreeinthekernel},  we are done just by considering the degree of $v$ in $K$. 

Assuming this is not the case,  Lemma \ref{lem:sparse-conductance} 
implies that whp for every vertex $v$ of $High$, $d_K(v) \ge \frac{d(v)}{C}$ so we are done unless $d(High) \le \frac{\epsilon'm^{\neq 2}}{64}$
and hence unless $d(V(K)-High) \ge \frac{7\epsilon' m^{\neq 2}}{64}$.  

Let ${\mathcal F}_i$ be the family of $G$ for which $d(V(K)-High) \ge \frac{m^{\neq 2}}{16}$ and there are $i$
edges of $K$ joining vertices of $V(K)-High$. There are  at least $
(\frac{\epsilon' m^{\neq 2}}{16}-i)(\frac{\epsilon' m^{\neq 2}}{16}-i-C\log m^{\neq 2})/2$
switchings from ${\mathcal F}_i$ to ${\mathcal F}_{i+1}$ using two edges 
from vertices in $V(K)-High$ to vertices in decorations hanging off them. 
There are at most $2(i+1)m^{\neq 2}$ swaps in the other direction.
For $m^{\neq 2}$ sufficiently large this implies that for $i  \le 
\left(\frac{\epsilon'}{64}\right)^2 m^{\neq 2}$, we have $\frac{|{\mathcal F}_i|}{|{\mathcal F}_{i+1}|} \le \frac12 $ and  hence
$\p{\bigcup_{i \le  \left(\frac{\epsilon'}{64}\right)^2 m^{\neq 2}} {\mathcal F}_i}=o(1)$. 

\end{proof}

\subsubsection{Bounding the number of edges  leaving a large set of vertices in the kernel}

 By  Lemma \ref{d(K)lower bound}, whp  $d(K) \ge \mu m^{\neq 2}$. 

In this section, we  bound $\out(W)$ for $W$ a subset of vertices in the kernel with $d(W) <\frac{d(K)}{2}$. 

\begin{lem}
\label{kernelbound}
    Whp every $W \subset K $ with $d(W) \le \frac{d(K)}{2}$  satisfies $\out(W) \ge  \frac{d(W)}{\log^{1/3} d(K)}$.  
\end{lem}

\begin{proof}
Applying  \Cref{d(K)lower bound}, it is sufficient to prove the 
lemma conditioned on $d(K) \ge \gamma m^{\neq 2}$, and we impose this conditioning. 
We now switch on  the ordered paths representing the edges of the kernel 
and  consider  a matching  $\cM^K$ of size $|E(K)|$  on the halfedges analogous to $\cM^*$. Note that the unicyclic components,
the vertices of the kernel,
their degrees in the kernel,  and $V(\cM^K)$ are fixed by such a matching. The decorations and which vertices of the kernel they hang off
and the interior of the decorated paths not including the edges incident to vertices of the kernel are also 
fixed (and in particular, we know which edges of $\cM^K$ correspond to edges of $G$ although this is not fixed).

Let $\cC_w = \{W \subseteq K : |W| = w, d(W) \leq \frac{d(K)}{2},\ \out(W)<\frac{d(W)}{\log^{1/3} d(K)}\}$,
and let $X_w = |\cC_w|$. To prove the lemma it is enough to show that $\sum_w \E {X_w}=o(1)$. In order to do so, we bound 
$\p{W \in \cC_w}$ for each $W \subseteq K$ with $|W|=w$.

We compute this probability by summing up the conditional probability that $W \in \cC_w$ given the specific choice of halfedges to be matched between $W$ and $K-W$; call these {\em boundary halfedges}.
There are at most $\sum_{j<\frac{d(W)}{ \log^{1/3} d(K)}} {d(W) \choose j}<2^{\frac{d(W)}{100}}$ choices for boundary halfedges incident to vertices of $W$. For any such set of halfedges, the conditional probability that $W \in \cC_w$ is the 
probability that each remaining halfedge incident to a vertex of $W$ is matched to another halfedge of $W$. We 
claim and prove below that whp  this probability is at most  $2^\frac{d(W)}{50}Z_W$ for 
$$Z_W=  \frac{(1)(3)...(2\lceil d(W)/2 \rceil-1)}{(2\lceil d(W)/2 \rceil+1)(2\lceil d(W)/2 \rceil+3)...(4\lceil d(W)/2 \rceil -1)}\ .$$
and at most  $2^\frac{d(W)}{50}(Z_W')^{99/100}$ for 
$$Z_W'=  \frac{(1)(3)...(2\lceil d(W)/2 \rceil-1)}{(d(K)-2\lceil d(W)/2 \rceil+1)(d(K)-2\lceil d(W)/2 \rceil+3)...(d(K)-1)}\ .$$

Assuming the claim holds, we may now bound $\sum_w \E{X_w}$. 
We treat separately $X'_w$ which sums over $W$ with $d(W) \ge \log \log d(K)$ and $|W|=w \ge k/2$ where $k = |V(K)|$,
$X^+_w$ which sums over $W$ with $d(W) \ge \log \log d(K)$ and $|W|=w < k/2$,
and $X^*_w$  which sums over $W$ with $|W|=w$ and $d(W) < \log \log d(K)$.

Let $\cS_w = \{W \subseteq K : |W| = w, d(W) \geq \log \log d(K)\}.$ First note that
$$\E{X'_w} \leq \sum_{W \in \cS_w} 2^{d(W)/50}Z_W\ .$$

If $2w \ge k$ and $d(W) \ge \log \log d(K)$, to bound $Z_W$, we observe that
$$Z_W = \frac{((2\lceil d(W)/2 \rceil-1)!!)^2}{(4 \lceil d(W)/2 \rceil-1)!!}\ .$$
Using the formula $(2d-1)!! = \frac{(2d)!}{2^d d!}$, 
and that the number of ways of choosing $a$ elements of $\{1,\ldots,2a\}$ and $a$ elements of $\{2a+1,\ldots,4a\}$ is at most the number of ways of choosing $2a$ elements from $\{1,\ldots,4a\}$, implying ${2a \choose a}^2 \le {4a \choose 2a}$,  
we obtain that $$\frac{((2a-1)!!)^2}{(4a-1)!!} = \frac{{2a \choose a}}{{4a \choose 2a}} \leq \frac{1}{{2a \choose a}}\, .$$

 Thus, 
$$\E{X'_w} \leq \sum_{W \subseteq K, |W|=w, d(W) \ge \log \log d(K)} 2^{d(W)/50}\frac{(2 \lceil (d(W)/2 \rceil)!)^2}{\lceil d(W) \rceil!}\leq {k \choose w} 2^{-48d(W)/50}\ .$$
We conclude this case by applying    the fact that $d(W) \geq 3w$ by definition of the kernel and the fact that $k \le 2w$. Since $w \ge k/2$, we have ${k \choose w} \le 2^{2w}$, so we can write 
\begin{align*} \E{X'_w} &\leq 2^{2w-47d(W)/50 - \log \log d(K)/50}\\
    &\leq 2^{-w/2}(\log \log d(K))^{-1}\ .
\end{align*}
Since $d(K)=\omega(1)$  we obtain $\E{\sum_w \ X'_w} =o(1)$.

For the case of $2w < k$ and $d(W)  \ge \log \log d(K)$, the computations proceed analogously to the previous case. We first write
$$\E{X^+_w} \leq \sum_{W \in \cS_w, {k \choose w} > 2^{d(W)/20}} 2^{d(W)/50}(Z_W')^{99/100} + \sum_{W \in \cS_w, {k \choose w} \leq 2^{d(W)/20}} 2^{d(W)/50} Z_W\ .$$
Using the definition of $Z_W'$, the bounds $d(K) \ge 2d(W)$ and $d(W) \geq 3w$, and the fact that the number of ways to choose $w$ elements from each of the sets $\{1, \dots, k\}, \{k+1, \dots, 2k\}$, and $\{2k+1, \dots, 3k\}$ is at most the number of ways to choose $3w$ elements of $\{1, \dots, 3k\}$, implying
${k \choose w}^3<{3k \choose 3w}$, we obtain that
\[
Z_W' \leq \frac{1}{\sqrt{{3k \choose 3w}}} \leq \frac{1}{{k \choose w}^{3/2}}\, .
\]
Using our previous bounds on $Z_W$, we then have
\begin{align*}\E{X_w^+} &\leq \sum_{W \in \cS_w, {k \choose w} > 2^{d(W)/20}}2^{d(W)/50}{k \choose w}^{-297/200} + \sum_{W \in \cS_w, {k \choose w} \leq 2^{d(W)/20}} 2^{-d(W)/50}2^{-d(W)/2}\\
&\leq 2^{-d(W)/1000} + 2^{-d(W)/3}\\
&\leq 2^{1-d(W)/2000} \cdot 2^{-3w/2000}
\end{align*}

Since $d(K)=\omega(1)$, by our assumptions on $w$ for this case, we have $w = \omega(1)$ as well and obtain $\E{\sum_w \ X^+_w} =o(1)$.

To compute $\sum_w E[X^*_w]$,  we proceed in a similar manner, exploiting the fact that $2^\frac{d(W)}{50}<\log d(K)$. Let $\cS'_w = \{W \subseteq K : |W| = w, d(W) < \log \log d(K)\}$. For $W \in \cS'_w$, note that 
\begin{align*} Z_W' &\leq \left(\frac{d(W)-1}{d(K)-d(W)+1}\right)^{d(W)/2} \\
&\leq \left(\frac{\log \log d(K)}{d(K) - \log \log d(K)}\right)^{d(W) / 2}\\ 
&\leq (d(K)^{-8/9})^{d(W)/2} = d(K)^{-4d(W)/9}\ .
\end{align*}
Using this and the fact that $d(W) \geq 3w$, we have
\begin{align*} \E{X^*_w} &\le  \sum_{W \in \cS_w'} \log d(K)(Z_W')^{99/100}\\
&\leq {k \choose w}   (\log d(K)) d(K)^{-11d(W)/25}\\
& \le k^w  (\log d(K)) d(K)^{-33w/25 } \\
&\le (\log d(K)) d(K)^{-8w/25 }\ .
\end{align*}

Since $d(K)=\omega(1)$ and $w \geq 1$, we obtain $\sum_w E[X^*_w]=o(1)$.

It remains to prove our claim.  Conditioned on our choice of $j$ boundary halfedges and their partners in $\cM^K$, we perform $\frac{d(W)-j}{2}$ iterations in which, for some halfedge hanging off a vertex in $W$ which is not yet matched, we
expose its partner in $\cM^K$.

We make a second claim  that given our initial conditioning and the outcome of the first $i-1$   iterations, the probability that the halfedge $h_i$ considered in the $i^{th}$ iteration matches with a halfedge hanging off a vertex in  $W$ is at most $\frac{1001 (d(W)-j-2i+1)}{1000(d(K)-2j-2i+1)}$.

We prove this second claim by showing that the probability  $h_i$ matches with any other halfedge $h'$  hanging off a vertex in $W$
is at most $\frac{1001}{1000(d(K)-2j-2i+1)}$. Let ${\mathcal F}_{yes}$ (respectively $\cF_{no}$) be the collection of matchings $M \in \cM^K$ corresponding to graphs $G$  satisfying our conditioning including  the set of potential boundary halfedges and outcome of the first $i-1$ iterations with $h$ matched to $h'$ (respectively, $h$ not matched to $h'$).

We consider switchings between matchings in ${\mathcal F}_{yes}$ and ${\mathcal F}_{no}$ by swapping the endpoints of two decorated
paths.  There are at most 2 such switchings from an element of ${\mathcal F}_{no}$ to elements of ${\mathcal F}_{yes}$. We note that $d(W) > j+2i$ so $d(K)-2i-2j
>d(K)-d(W)-j>d(K)/3$. So, 
since  the degree sequence has no  $\mu$-center, and $d(K) \ge \gamma m^{\neq 2}$ by Lemma~\ref{d(K)lower bound}, there are at least $2(d(K)-2j-2i+1-\frac{d(K)}{6(10^6)}) >(2-10^{-6})(d(K)-2j-2i+1)  $ switchings from an element of ${\mathcal F}_{yes}$ to elements of ${\mathcal F}_{no}$. Applying \eqref{switchingineq}, this proves our second claim.

It remains to show that  for   $j \le \frac{d(W)}{\log d(K)^{1/3}}$  of the same parity as $d(W)$, we have: 
$$\prod_{i=1}^{(d(W)-j)/2} \frac{1001 (d(W)-j-2i+1)}{1000(d(K)-2j-2i+1)}\le 2^{d(W)/50}\min\{Z,(Z')^{99/100}\}$$
which follows if we show that  
$$\prod_{i=1}^{(d(W)-j)/2} \frac{ d(W)-j-2i+1}{d(K)-2j-2i+1}\le 2^{d(W)/100}\min\{Z,(Z')^{99/100}\}\,.$$

Now  for all $i \leq \frac{d(W)-j}{2}$, note that $d(K)-2j-2i+1>d(W)-j+1 \ge \frac{d(W)}{2}$. 
Since
$j \le \frac{d(W)}{(\log d(K))^{1/3}}$ we then have 
$$\prod_{i=1}^{(d(W)-j)/2} \frac{ d(K)-2i+1}{d(K)-2j-2i+1}\le 
\left(1+\frac{4j}{d(W)}\right)^{(d(W)-j)/2} \le 2^{d(W)/200}.$$

So, it is enough to show
\begin{equation}\label{eqn:boundingZZ'}
\prod_{i=1}^{(d(W)-j)/2} \frac{ d(W)-j-2i+1}{d(K)-2i+1}\le 2^{d(W)/200}\min\{Z,(Z')^{99/100}\}\ .
\end{equation}

Recall the definitions of $Z_W$ and $Z_W'$:
$$Z_W=  \frac{(1)(3)...(2\lceil d(W)/2 \rceil-1)}{(2\lceil d(W)/2 \rceil+1)(2\lceil d(W)/2 \rceil+3)...(4\lceil d(W)/2 \rceil -1)}\ ,$$
$$Z_W'=  \frac{(1)(3)...(2\lceil d(W)/2 \rceil-1)}{(d(K)-2\lceil d(W)/2 \rceil+1)(d(K)-2\lceil d(W)/2 \rceil+3)...(d(K)-1)}\ .$$
We may expand the left-hand side of $\eqref{eqn:boundingZZ'}$ as
$$L := \frac{1(3)\cdots(d(W)-j-1)}{(d(K)-d(W)+j+1)(d(K)-d(W)+j+3)\cdots(d(K)-1)}\ .$$
It is clear, then, that the numerator of $L$ is obtained by deleting some $j'$ largest terms from the common numerator of $Z$ and $Z'$, where $j'-\frac{j}{2} <2$.
Similarly, the denominator of $L$ is obtained by deleting $j'$ smallest terms from the denominator of $Z'$.  Since there are more than $100j'$ terms in the denominator, $L$ is therefore at most $(Z')^{99/100}$. On the other hand, since $d(K)$ is even, the denominator of $L$ is also obtained by deleting the $j'$ smallest terms from the denominator of $Z$ and adding some even amount $a \geq -2$ to each remaining term. 
Since the denominator terms of $Z$ are all at least $d(W)$ and there are fewer than
$d(W)$ of them, shifting one term of the denominator multiplies the value of the product by at most 3.  Since the $j$ smallest terms of the denominator of $Z$ are 
all at most $d(W)+2j+3<\sqrt{3}d(W)$ and the $j$ largest  terms of the numerator  of $Z$ are 
all at least $d(W)+2j+3>\frac{d(W)}{\sqrt{3}}$, it follows that $L$ is at most $Z3^{j'+1}<2^{d(W)/200}Z$. 
\end{proof}

\begin{cor}
\label{kernelalltogether}
    Whp there is exactly  one component of $G$ which contains two or more cycles, and this component contains the kernel.   
\end{cor}

\begin{proof}
    \Cref{kernelbound} implies the kernel is connected and thus lies in one component whp.
    Lemma \ref{d(K)lower bound} implies the kernel is non-empty. 
    It is not hard to see that a  connected component contains two or more cycles if and only if it intersects the kernel. 
 \end{proof}

\subsection{Putting it all together}

We now translate our results about $J$ to results about $G$,
in order to prove Theorem \ref{thm:mainnew5}.  Recall that for $S \subset V(G)$, we use $F_S$ to denote the subgraph of the coloured multigraph $J$ whose vertices are those of $J$ contained in $S$ and whose coloured edges are those for which the entirety of the corresponding path is contained in $S$. Using our results from the previous section on subgraphs of $J$, we may obtain the following bound for subsets of $V(G)$.

\begin{prop}
\label{saulgoodman}
    Whp there  is a  $C$  such  that every $S \subset V(G)$ inducing a connected subgraph of $G$  with $d(S)> \frac{Cm \log m^{\neq 2}}{m^{\neq 2}}$ satisfies $d(S) \le \frac{4Cd(F_S)m}{m^{\neq 2}}$.
\end{prop}

\begin{proof}
     Given a subgraph $F$ of $J$, let $F^*$    be the union of  $V(F)$ and the interior of the paths of the edges 
     of $J$ incident to $V(F)$. Then any set $S$ is contained in $F^*_S$. 
     So, it is enough to show that the expected number of connected 
     subgraphs  $F$ in  $J$ for which $d(F^*)>\max\left\{\frac{4Cd(F)m}{m^{\neq 2}},  \frac{Cm \log m^{\neq 2}}{m^{\neq 2}}\right\}$ is $o(1)$.

    Applying \Cref{treenumber} the expected number of choices for a submatching of $M^*$ yielding  a  spanning tree of a fixed subgraph $F$  of $J$ with $d(V(F))=d$
   and $\ex(V(F))=\ex$  is less than $m^{\neq 2}{d \choose \ex}$. Summing over all choices of $\ex$, then, the expected  number of  choices of submatching of $M^*$  yielding a tree spanning a subgraph $F$ of   $J$ with $d(V(F))=d$ is less than $m^{\neq 2}2^d$. We show that,  conditioned on the existence of any  fixed such submatching and the subgraph $F$ it specifies, the probability that $d(F^*)> \max\left\{\frac{4Cd(F)m}{m^{\neq 2}},  \frac{Cm \log m^{\neq 2}}{m^{\neq 2}}\right\}$ is $o(\frac{1}{m^{\neq 2}d^22^d})$ thereby proving the proposition.
   
   We will actually prove this bound conditioned on a full choice of $M^*$ extending the  submatching. This allows us to determine the choice of the set of at most $d$ edges of $M^*$ which are incident to vertices of $F$. Of course, $d(F^*)$ is simply  the sum of $d(F)$ and twice the  length of all the paths corresponding to these edges. We then condition further on the set $Z$ of the at-most-$d$ edges incident to $F$ which are green.
   Since the red and yellow paths correspond to paths of total length at most $2d$, we need to show that under this conditioning, the probability of the event $E_Z :=$ \Big\{the sum of the length 
   of the paths corresponding  to the edges of $Z$ exceeds 
   $\max\left\{\frac{4Cdm}{m^{\neq 2}}-4d,  \frac{Cm \log m^{\neq 2}}{2m^{\neq 2}}\right\}$\Big\} is $o\left(\frac{1}{m^{\neq 2}d^22^d}\right)$.

   We bound  separately the probability of    $E_Z \cap \{g>\frac{m^{\neq 2}}{100}\}$   and $E_Z \cap \{g \le \frac{m^{\neq 2}}{100}\}$.
   
To bound $E_Z \cap \{g>\frac{m^{\neq 2}}{100}\}$ we further condition on  the set of  all green edges.   By Proposition \ref{ballsinbins}, for $N = n_2 - 2g - y$ and letting  
     $$X=\mathrm{Bin}\left(g-1,\frac{1}{N}\left(\frac{\max\{2Cd,\frac{C}{2}\log m^{\neq 2}\}m}{m^{\neq 2}}\right)\right)\ ,$$
     we have  $\p{E_Z} \leq \p{X<d}$. Since $N<m$,  whp we have 
     $$\E{X} > \max\left\{\frac{Cd}{100}, \frac{C \log m^{\neq 2}}{200}\right\}-4d>40d+40\log m^{\neq 2}.$$
     Furthermore, $\p{E_Z} \leq \p{X<\frac{\E{X}}{2}}$, which is at most $2^{-\e {[X]}/12}<2^{-2d}(m^{\neq 2})^{-2}$.

We consider next the possibility that $ g\le \frac{m^{\neq 2}}{100}$ and $r>\frac{m^{\neq 2}}{3}$. For $i \le 40d+40\log m^{\neq 2} <\frac{m^{\neq 2}}{100}$,
    let ${\mathcal F}_i$ be the family of graphs $G$ for which $ g \le \frac{m^{\neq 2}}{100}+i$ and $r>\frac{m^{\neq 2}}{3}-i$ and 
     the sum $\sigma$  of the length of the edges of $Z$ exceeds 
     $\frac{\max\{Cd,C\log m^{\neq 2}/2\}m}{m^{\neq 2}}-2i$. So
     $E_Z \cap \{g\le \frac{m^{\neq 2}}{100}\} \cap \{r>\frac{m^{\neq 2}}{3}\} \subseteq {\mathcal F}_0$. 

We can move from a graph in ${\mathcal F}_i$ to  a graph
      in ${\mathcal F}_{i+1}$ by taking any two vertices in green 
      edges of $S$ with at least four internal vertices and using them to 
      subdivide a red edge.  There are at least $(\frac{m^{\neq 2}}{3}-i)\frac12(\sigma-2i-4d)^2>\frac{m^{\neq 2}\sigma^2}{4}$  ways of doing this.
      To swap in the other direction, we use the two  internal vertices of some green edges to subdivide edges corresponding to $Z$. There are at most $(\frac{m^{\neq 2}}{100}+i)(\sigma+d+1)^2<\frac{m^{\neq 2}\sigma^2}{40}$
      ways of doing this. Thus, $|\cF_i| \le \frac{|\cF_{i+1}|}{10}$, 
      and 
      \[
      \p{E_Z \cap (g\le \frac{m^{\neq 2}}{100}) \cap (r>\frac{m^{\neq 2}}{3})} \le  \p{{\mathcal F}_0} \le 10^{-40d-40\log m^{\neq 2}}\, .
      \]

The remaining case is when $ g\le \frac{m^{\neq 2}}{100}$ and $r \le \frac{m^{\neq 2}}{3}$, so that $y \le \frac{m^{\neq 2}}{3}$. The argument is identical to the previous case except that in the switching, we move only one vertex from a green to a yellow edge.

\end{proof}

We now have all the ingredients to prove our main result for degree sequences without a center.

\begin{proof}[Proof of Theorem~\ref{thm:mainnew5}]

Combining  \Cref{decorationsize} and  \Cref{saulgoodman}, 
we obtain that whp the maximum size of a decoration, decorated path, or unicyclic component is $O(\frac{m\log m^{\neq 2}}{m^{\neq 2}})$.
Thus every component of $G$ disjoint from the kernel has size and hence diameter $O(\frac{m\log m^{\neq 2}}{m^{\neq 2}})$. Hence whp these components have  mixing time $O((\frac{m \log m^{\neq 2}}{m^{\neq 2}})^2)$.

Lemma \ref{kernelbound} implies that the kernel is connected, so it 
remains to  bound the  diameter
 and mixing time of the component  $K^*$ containing it   by considering the conductance of the connected sets  within it.  
We consider  each   subset $S$ of the vertices of $K^*$ forming a connected set with $d(S) \le d(K^*)/2$. 

Let $C_1$ be the constant given by \Cref{lem:sparse-conductance}, $C_2$ be the constant given by \Cref{saulgoodman}, and $C = C_1C_2$. If $d(S) \le \frac{4Cm \log m^{\neq 2}}{m^{\neq 2}}$
Proposition \ref{saulgoodman}  implies  that whp  every $S$ with  $\frac{4C m\log  m^{\neq 2}}{m^{\neq 2}} \le d(S) \le \frac{m}{\log \log m^{\neq 2}}$,  satisfies $C_1\log m^{\neq 2} <d(F_S)$.
By Proposition \ref{sumoflengthsprop2}, we have  that for any such $S$, $d(F_S)<\frac{m^{\neq 2}}{C_1}$. 
Hence   Lemma \ref{lem:sparse-conductance}
and Proposition \ref{saulgoodman} imply whp for all such $S$  we have $\cond(S)=\Omega(\frac{m^{\neq 2}}{m})$.

If $\frac{m}{\log \log m^{\neq 2}}<d(S) \le \frac{d(K^*)}{2}$ and   
$\out(S)=\out(V-S)>\frac{m^{\neq 2}}{\log^{6/7} m^{\neq 2}}$, then  $\cond(S) \ge \frac{m^{\neq 2}}{m\log^{4/5} m^{\neq 2}}$. 

We claim that whp, there is no  $S$ with  $\frac{m}{\log \log m^{\neq 2}}<d(S) \le \frac{m}{2}$  and $\out(S)=\out(V-S)<\frac{m^{\neq 2}}{\log^{6/7} m^{\neq 2}}$. 
To prove our claim we note that  by Lemmas  \ref{d(K)lower bound} and \ref{kernelbound}, whp for any such $S$,   either $d(K \cap S)<\frac{d(K)}{\log^{1/2} d(K)}$ or $  d(K \cap S) \le d(K)- \frac{d(K)}{\log^{1/2} d(K)}$.  
So, some $W \in \{S,K^*-S\}$ satisfies $d(W)> \frac{m}{\log \log m^{\neq 2}},\ d(K \cap W)<\frac{d(K)}{\log^{1/2} d(K)}$, and  $\out(W)
<\frac{m^{\neq 2}}{\log^{6/7} m^{\neq 2}}$.

For each decorated path or decoration  which contains an edge from $W$ to $V-W$  and either  contains more than $\frac{m\log^{6/7} m^{\neq 2}}{4m^{\neq 2}\log \log m^{\neq 2}}$ edges incident to $W$ or is incident to an edge which is also incident to $K \cap W$, we add the interior of the decorated path or the decoration to $W$. We still have  
$d(W)> \frac{m}{\log \log m^{\neq 2}}, \ \out(W)
<\frac{2m^{\neq 2}}{\log^{6/7} m^{\neq 2}}$, and   $d(K \cap W)<\frac{2m^{\neq 2}  }{\log^{1/2} m^{\neq 2}}$. 

For  each decorated path or decoration  which contains an edge from $W$ to $V-W$ but is incident to no edge which is also incident to $K \cap W$ and contains fewer than 
$ \frac{m\log^{6/7} m^{\neq 2}}{4m^{\neq 2} \log \log m^{\neq 2} }$ edges incident to $W$, we  delete from $W$ its intersection with the decorated path or decoration.   We still have $\out(W)<\frac{2m^{\neq 2}}{\log^{6/7} m^{\neq 2}},\ d(K \cap W)<\frac{2m^{\neq 2}}{\log^{1/2} m^{\neq 2}}$, and (because of our bound on $\out(W)$) $d(W)> \frac{m}{2\log \log m^{\neq 2}}$.

We note that   every decorated path and decoration  now either lies entirely in $W$ or entirely outside $W$. 
Furthermore, there are at most $d(K \cap W)+\out(W)<\frac{3m^{\neq 2}}{\log^{1/2} m^{\neq 2}}$ decorated paths and decorations  which  lie entirely within $W$. 
Applying Lemma \ref{decorationsize2} with $\ell=(\log m)^{1/4}$ we see that whp  for any choice of such a $W$, the number of edges of $J$ on  these paths is 
$o(\frac{m^{\neq 2}}{\log^{1/7} m^{\neq 2}})$. But, $d(W)> \frac{m}{2\log \log m^{\neq 2}}$ and so applying \Cref{saulgoodman}, our claim is proved.

Summing up, we have the following whp. 
\begin{enumerate}
    \item If $d(S) \le \frac{4Cm \log m^{\neq 2}}{m^{\neq 2}}$  and $S$ is not a component, then $\cond(S)>\frac{1}{d(S)}$.
    \item If $\frac{4C m\log  m^{\neq 2}}{m^{\neq 2}} \le d(S) \le \frac{m}{\log \log m^{\neq 2}}$, then $\cond(S)=\Omega(\frac{m^{\neq 2}}{m})$. 
     \item If $\frac{m}{\log \log m^{\neq 2}}<d(S) \le m$,  then  $\cond(S)>\frac{m^{\neq 2}}{m\log^{4/5} m^{\neq 2}}$. 
\end{enumerate}

Applying Observation \ref{conductanceboundnew} and Lemma \ref{conductancebound}, we see that the whp the diameter of every component is $O(\frac{m\log m^{\neq 2}}{m^{\neq 2}})$   and the mixing time of every component is   $O((\frac{m\log m^{\neq 2}}{m^{\neq 2}})^2)$.
\end{proof}

\section{Degree Sequences with a Center}\label{sec:withcenter}

It remains to prove Theorem \ref{thm:mainnew1}
for degree sequences with a $\mu$-center, which we do in this section.
As was the case for degree sequences without a $\mu$-center, applying Proposition
\ref{fewcycles} allows us to restrict our attention to non-cycle components.

We recall that degree sequences with a $\mu$-center satisfy $m^{\neq 2}\geq \mu^3 m$ and  $\sum_{i=n-d_n+1}^n d_i \ge \mu^2  m^{\neq 2} (\ge \mu^5 m)$. 
We note that this implies that $d=d_n \ge \mu^{2.5} \sqrt{m}$
and that $\frac{m}{m^{\neq 2}}<\mu^{-3}$. 
Let $S_{\ge i}$ be the set of vertices of degree at least $i$.
We use $S$ for the vertices of $S_{\ge\lambda \sqrt{m}}$
for some $\lambda$ sufficiently small in terms of $\mu$ and $\rho$.

Recall that $B_i(v)$ is the set of vertices joined to $v$ by a path 
with at most $i$ edges. We first bound the distance from any vertex to the vertex with maximum degree. 

\begin{lem}
\label{lastlem} 
For any sequence of degree sequences with a $\mu$-center and any function $g$ going to infinity with $m$, whp every vertex of $S$ is in  $B_{<g(m)}(n)$. 
\end{lem}

\begin{proof}
We can and do assume $g(m) \le \log \log m$. 
We set $f(m)=g(m)^{1/12}$, so we need to prove whp every vertex of $S$ is in  $B_{f(m)^{12}}(n)$. The outline of the proof is the following. If $d_n \ge \sqrt{m}f(m)$,
then a simple switching  argument shows that the probability a vertex of $S$ is not adjacent to $n$ is $o(1)$; this allows us to conclude that whp most of the vertices of $S$ are adjacent to $n$. It is then a simple matter to show  via a second switching argument that all of $S$ is within distance 4 of $n$. If $d_n \le f(m)\sqrt{m}$, then because the degree  sequence is centered there are many vertices of degree exceeding
$\sqrt{m}/f(m)^2$. We show that a significant proportion of them are within distance $9$ of $n$. We then use this clump of vertices to show almost all of the vertices of $S$ are within distance 15 of $n$. We can then complete the proof.

{\bf Case 1:} $d_n \ge \sqrt{m}f(m)$. 

In this case for any vertex $v$ of $S$, there are  more than $(\lambda f(m)-2)m$ switchings from any graph in which $vn \not \in E$  to graphs in which $vn \in E$, and at most $2m$ switchings from any graph in which  $vn  \in E$  to 
graphs in which $vn  \not \in E$.  So, $\p{vn \notin E} < \frac{4}{\lambda f(m)}$.
It follows that if $|S| \le \frac{f(m)}{\log \log f(m)}$, then with high probability all of $S$ is adjacent to $n$, and we are done. Else, we assume that $|S| \ge \frac{f(m)}{\log \log f(m)}$. The expected number of vertices of $S$ nonadjacent to $n$ 
is then at most $\frac{4|S|}{\lambda f(m)}$, so by Markov's inequality, whp  $\frac{9|S|}{10}$ vertices of $S$ are in $N(n)$. 

For each vertex $v$ of $S$, we let $Bad_v$ be the event  that $|N(n) \cap S| \ge \frac{9|S|}{10}$ and $v \not \in B_4(n)$. For every $H \in Bad_v$ there are at least $\frac{\lambda^2 m|S|}{2}$ switchings 
from $H$ to graphs not in $Bad_v$  where $n$ and $v$ are non-adjacent but share exactly one common neighbour, using an edge $uv$ and an edge $yx$ for some $y \in N(n) \cap S$ and $x \neq n$. For every graph $H' \not \in Bad_v$ in which $n$ and $v$ have a unique common neighbour, there are at most $2m$ switchings using the edge from $v$ to the common neighbour. 
So, since $\p{vn \not \in E} \leq \frac{4}{\lambda f(m)}$, we then have $\p{Bad_v} \le \frac{16}{\lambda^3 f(m) |S|}$.
Thus, whp no $Bad_v$ occurs and whp $S \subseteq B_4(n)$.

{\bf Case 2:} $d_n \le \sqrt{m}f(m)$. 

Since the degree sequence is centered, we know that $\sum_{i : d_i > \frac{\mu^5 m}{2d_n}} d_i \geq \frac{\mu^5 m}{2}$, and hence there must be at least $\frac{\mu^5 m}{2d_n}>\frac{\sqrt{m}}{f(m)^2}$
vertices of degree at least $\frac{\mu^5 m}{2d_n}$. We let $S'$ be the set of vertices of degree  at least  $\frac{\mu^5 m}{2d_n}>\frac{\sqrt{m}}{f(m)^2} $.
Then $\frac{\sqrt{m}}{f(m)^2} \le |S'| \le 2\sqrt{m}f(m)^2$ and $d(S')>\frac{\mu^5 m}{2}$. 
It is enough to show that every vertex of $S'$ is in $B_{f(m)^{12}}(n)$.

\begin{claim}
\label{lastclaim}
  Whp either $|B_3(n)| \ge  \frac{m^{1/2}}{(\log m)^{10}}$  or for some $v \in S'$,  $|B_3(v)  \cap S'| \ge  \frac{m^{1/2}}{(\log m)^{10}}$.    
\end{claim} 

\begin{claimproof}
    We let ${\mathcal F}_{i,j}$ be the family of graphs $H$ that satisfy the following properties: 
\begin{enumerate}
    \item $n$ is adjacent to exactly $i+j$ 
    vertices of $S'$,
    \item no $v \in V-\{n\}$ is adjacent to more than $\frac{m^{1/2}}{(\log m)^{10}}$  vertices of $S'-\{n\}$, 
    \item  every  vertex  $v$ of
    $S'-\{n\}$  is at distance at most 3 in $H-\{n\}$ to at most  $(2j+1)\frac{m^{1/2}}{(\log m)^{10}}$  
vertices of $S'$, and
    \item $|B_3(n) \cap S'| \le \frac{(j^2+2j+1)m^{1/2}}{(\log m)^{10}}$. 
\end{enumerate}      
Note that it is enough to prove that $$\p{\bigcup_{i=0}^{\lfloor \frac{m^{1/2}}{(\log m)^{10}}  \rfloor} {\mathcal F}_{i,0}}=o(1)\ .$$
We do this by showing that for each such $i$, we must have
$\p{{\mathcal F}_{i,0}}<\frac{\p{{\mathcal F}_{i,\lceil \log m \rceil}}}{m}
<\frac{1}{m}$. 

There are at least $\frac{\mu^5 m}{2}$ edges incident to $S'$.  For any graph $H$ in ${\mathcal F}_{i,j}$ 
for $i  \le  \frac{m^{1/2}}{(\log m)^{10}}$ and $j \le \lceil \log m \rceil$, at most 
$|B_3(n) \cap S'| d_n <\frac{\mu^5 m}{4}$ of these are incident to vertices in 
$B_3(n) \cap S'$. Thus, there are at least $\frac {\mu^5 m^{3/2}}{f(m)^3}$ switchings using edges $nu$ and $xy$ where 
$u,x \not \in S'$ and $y \in S'-B_3(n)$. 

We verify that the result of such a switching is in the family $\cF_{i,j+1}$. Regarding (1), any such switching increases $|N(n) \cap S'|$ to $i+j+1$. 
Regarding (2), for any $v$ in $V-\{n\}$ the  number of neighbours of $v$ in $S'-\{n\}$ does not increase. Regarding (3), for any $v \in S'-\{n\}$ (including $v = y$), the number of vertices  of $S'$ of distance at most three from $v$ in $H-\{n\}$ increases by at most $|N(u)\cap (S'-\{n\})|+|N(x) \cap (S'-\{n\})|\le 2 \frac{m^{1/2}}{(\log m)^{10}}$ to at most $(2j+3) \frac{m^{1/2}}{(\log m)^{10}}$. 
Regarding (4), the new vertices in  $B_3(n) \cap S'$  are either vertices that were in $B_3(y) \cap S'$ or are neighbours of $u$ or $x$. So, $|B_3(n) \cap S'|$  increases by at most $(2j+3) \frac{m^{1/2}}{(\log m)^{10}}$ to at most $\frac{((j+1)^2+2(j+1)+1)m^{1/2}}{(\log m)^{10}}$. 
The result of such a switching is thus a graph in $\cF_{i,j+1}$. 

In the other direction, there are at most $2(i+j+1)m$ switches from a graph in ${\mathcal F}_{i,j+1}$
to  a graph in ${\mathcal F}_{i,j}$, So,  $|{\mathcal F}_{i,j}| \le \frac{|{\mathcal F}_{i,j+1}|}{\log m}$
and $\p{{\mathcal F}_{i,0}}<\frac{\p{{\mathcal F}_{i,\lceil \log m \rceil}}}{(\log m)^{\log m}}
<\frac{1}{m}$, proving the claim.
\end{claimproof}

Recall that our goal is to show that whp, the ``bad'' cases where there is a vertex in $S'$ far from $n$ do not occur. We now do this by defining a sequence of bad events which allow us to iteratively refine our understanding of $G$, first using $B_9(n)$, then $B_{12}(n)$, and finally $B_{15}(n)$. At each step, we use the switching inequality to show that the bad cases of the current step contribute $o(1)$ in probability.

First, let $Bad$ be those graphs for which  $|B_9(n)|<\frac{\sqrt{m}}{(\log m)^{20}}$ but there is a $v \in S'$ with $|B_3(v) \cap S'| \ge \frac{\sqrt{m}}{(\log m)^{10}}$.
Given $H \in Bad$, let $v$ be such a vertex,
so $v \not \in B_6(n)$. We let
$T$ be a tree  of height 3  rooted at $v$ which contains exactly $\lceil \frac{\sqrt{m}}{(\log m)^{10}} \rceil$  vertices of $B_3(v) \cap S'$ and whose leaves are in $B_3(v) \cap S'$. Then $|T| \le \lceil \frac{3\sqrt{m}}{(\log m)^{10}}\rceil$. We consider switchings from $H$ to graphs that satisfy $|B_9(n)| \geq \frac{\sqrt{m}}{(\log m)^{20}}$ and thus are not in $Bad$.
There are at least $\frac{m^{3/2}}{(\log m)^{12}}$ switchings from $H$ to a graph outside $Bad$  using an edge $xy$ not in $T$ with $ y \in T \cap S'$  and an edge $nz$ for some $z \not \in S'$.
In the other direction, since $v$ has at most $\frac{\sqrt{m}}{(\log m)^{20}}$ neighbours in $S'$ in any graph in $Bad$, for any graph $H'$ not in  $Bad$ there are at most $2m(\frac{\sqrt{m}}{(\log m)^{20}}+1)$
switchings from $H'$ to $Bad$  using an edge from $n$ to a vertex  of $S'$. By applying \eqref{switchingineq}, we can see that whp $G(\cD) \not \in Bad$.
Combined with Claim \ref{lastclaim} this tells us that whp $|B_9(n)  \cap S'|\ge \frac{\sqrt{m}}{(\log m)^{20}}$.  

We next consider the event  $Bad'_v$  that $|B_9(n)| \ge \frac{\sqrt{m}}{(\log m)^{20}}$  and $v \in S'-B_{12}(n)$. For each $H \in Bad'_v$, we let
$T$ be a tree  of height 10  rooted at $n$ which contains exactly $\lceil \frac{\sqrt{m}}{(\log m)^{20}}\rceil$ vertices of $B_9(n) \cap S'$ and whose leaves are in $B_9(n) \cap S'$. Then $|T| \le \lceil \frac{3\sqrt{m}}{(\log m)^{20}} \rceil$. We can and do ensure that for every $w \in T$, the path from $w$ to $n$ in $T$ is a shortest path from $w$ to $n$.

There are at least $\frac{m^{3/2}}{(\log m)^{22}}$ switchings from $H$ to a graph outside $Bad'_v$  using an edge $vz$ and an edge not in $T$---say $xy$---such that either $x$ or $y$ is in $T \cap S'$ and the ordering $xy$ is such that the  length of the shortest
path in $H$  from $n$ to $y$   is no longer  than the length of the shortest 
path  in $H$ from $n$ to $x$.
Since $\mathrm{dist}(z, n) \geq 11$ in $H$,
and neither $x$ nor $z$ are neighbors of  $v$ after the switching, it follows that $vy$ is the first edge of every  shortest path from $v$ to  $n$ in the new graph.
For any graph, there are at most $2m$ switchings that use the edge which is the first edge of every shortest path from $v$ to $n$. 
So $\p{Bad'_v}\le \frac{(\log m)^{22}}{\sqrt{m}}$.

Thus, $\p{\{|B_9(n)| \ge \frac{\sqrt{m}}{(\log m)^{20}}\} \cap \{|S'-B_{12}(n)|> \frac{|S'| (\log m)^{23}}{\sqrt{m}}\}}=o(1)$. Since 
$|B_9(n)| \ge \frac{\sqrt{m}}{(\log m)^{20}}$  whp, it follows that whp
$|B_9(n)| \ge \frac{\sqrt{m}}{(\log m)^{20}}$ and 
$|S'-B_{12}(n)| \le  \frac{(\log m)^{23}|S'|}{\sqrt{m}}< (\log m)^{24}$.

We next consider the event  $Bad^*_v$  that  
$|S'-B_{12}(n)| \le  (\log m)^{24}$  and $v \in S'-B_{15}(n)$.  We can mimic the argument above, using a tree  $T'$ containing $B_{12}(n) \cap S'$ 
in the place of $B_{9}(n) \cap S'$. Because the degree sequence is centered, there are at least $\frac{\mu^5 m}{4}$ edges 
incident to vertices of $T'$   which we can use to switch and hence at least $\frac{\mu^5 m^{3/2}}{4f(m)^2}$
switchings from a graph in $Bad^*_v$.  It follows that  $\p{Bad^*_v}\le \frac{4f(m)^2}{\mu^5 \sqrt{m}}$
and whp $|S'-B_{15}(n)| \le f(m)^5$.

Lastly, we consider the event $Bad^+_v$ that $|S'-B_{15}(n)| \le f(m)^5$ and $v \not \in B_{f(m)^{12}}(n)$.
Let $P$ be a shortest path from $v$ to $n$. For any  $H \in  Bad^+_v$,  we can switch any edge $xy$ of $P$ such that $x$ is closer to $v$
than $y$ and such that $\frac{f(m)^{11}}{2} \leq \mathrm{dist}(x,v) \leq f(m)^{11}$  with at least $\frac{\mu^5  m}{4}$
edges incident to vertices of $T'$ such that the new edge incident to $x$ is  one of the first 16 edges on every shortest $n-v$ path;  such a path has length  between $\frac{f(m)^{11}}{2}$ and $f(m)^{11}+16$, so $v \not \in B_{15}(n)$. For any graph, there are at most $16m$
switchings which use an edge that is one of the first 16 edges in every shortest path from $n$ to $v$.
Thus, 
$$\p{Bad^+_v} \le \frac{128 \p{\{|S'-B_{15}(n)| \le f(m)^5\} \cap \{v \not \in B_{20}(n)\}}}{\mu^5 f(m)^{11}}\ .$$

For any graph satisfying  $|S'-B_{15}(n)| \le f(m)^5$ and $v \not \in B_{20}(n)$, there are $\frac{\mu^5 m^{3/2}}{4f(m)^2}$ switchings to a graph not in this set 
using an edge incident to  $T'$ and an edge incident to $v$ so that the new edge from $v$ is the first edge on every shortest path 
from $v$ to $n$ in the new graph. Thus, $\p{\{|S'-B_{15}(n)| \le f(m)^5\} \cap \{v \not \in B_{20}(n)\}} \le \frac{4f(m)^2}{\mu^5 \sqrt{m}}$ and so  $\p{Bad^+_v} \le \frac{512 f(m)^2}{\mu^{10} \sqrt{m}f(m)^{11}}$.

As $|S'| \le \sqrt{m}f(m)^2$ it follows that with high probability no $Bad^+_v$ occurs. 
Since whp, $|S'-B_{15}(n)| \le f(m)^5$, this completes the proof. 
\end{proof}

To prove Theorem~\ref{thm:mainnew1}, we consider $\mu$-centered degree sequences in two separate cases: those whose high-degree vertices contribute a large fraction of the total number of edges and those which do not. 
We first show the following for the former case.

\begin{prop}
\label{dn-3large}
    Suppose we have a sequence of degree sequences with a $\mu$-center such that $d(S) > \rho\mu^{10}m$. If $d_{n-3}  \ge \frac{m}{D \log m}$, then whp every vertex of  
$S_{\ge \frac{n^{1/3}}{\log m}}-S$ is adjacent to a vertex of $S$.  
\end{prop}

\begin{proof}
  For every $v \in S_{\ge\frac{n^{1/3}}{\log m}}-S$ and $i \le 4$, let $E_{v,i}$ be the event that $v$ is adjacent to exactly  $i$ elements of $\{n-3,n-2,n-1,n\}$ and none of $S-\{n-3,n-2,n-1,n\}$.
For any $H \in E_{v,i}$ with $i \le 3$, every neighbour of $v$ outside 
$\{n-3,n-2,n-1,n\}$ is adjacent to less than half the neighbours of  all of 
$\{n-3,n-2,n-1,n\}$.   So, for $i \le 3$ there are more than $\frac{d(v)m}{4D \log m}$
choices for switchings from $E_{v,i}$ to $E_{v,i+1}$.
On the other hand there are at most $2(i+1)m$ switchings in the other direction.
So, $\p{E_{v,i}}=O(\frac{\log m }{d(v)}\p{E_{v,i+1})}$ and 
$\p{E_{v,0}}=O(\frac{(\log m)^4}{d(v)^4})=o(\frac{(\log m)^8}{n^{4/3}})=o(n^{-1})$. Summing over all $v$ finishes the proof.  
\end{proof}

\begin{lem}
\label{notlastlem} 
    Suppose we have a sequence of degree sequences with a $\mu$-center such that $d(S) > \rho \mu^{10}m$. 
Whp, for any $h$ going to infinity with $n$ and sufficiently large $D$  we have that 
\begin{enumerate} [(I)]
    \item if  $d(S-S_{\ge \frac{m}{ D \log m}}) >\frac{\rho \mu^{10} m}{2}$, then every vertex of $S_{\ge D\log m}$ is within distance 6 of $S$, and
    \item every vertex of $S_{ \ge n^{1/3}h(n)}$ is within distance 6 of $S$. 
\end{enumerate}
\end{lem}

\begin{proof}

Suppose first that for some sufficiently large $D$, we have $d(S-S_{\ge  \frac{m}{ D \log m}}) >\frac{\rho \mu^{10} m}{2}$. 
For any $v \in S_{\ge D\log m}-S$ and $i \le  \frac{\rho\mu^{10} D \log m}{100} \le \frac{\rho \mu^{10}d(v)}{100}$, let ${\mathcal F}_i$ 
be the family of graphs where $v$ is adjacent to  at most 
$2i$ vertices which are either  in $S$ or have more than $2i$ 
neighbours which are adjacent to some vertex of  $S$.  Note that $\cF_0$ contains the event that $\mathrm{dist}(v,S) \geq 4$.

For any  graph in  ${\mathcal F}_i$, the number of switchings using edges $vw$ and $zs$ such that
\begin{itemize}
 \item $w \in N(v) - S$, 
 \item $w$ has at most 
$2i$ neighbours in common with vertices of $S$,
\item $s \in S-S_{\ge \frac{m}{D \log m}}-N(v)$, and 
\item $z \in N(s)-N(w)$
\end{itemize}
is at least $$(d(v)-2i)\left( \left(\frac{\rho\mu^{10}m}{2}-\frac{2im}{D \log m}\right) \frac{\lambda \sqrt{m}-2i}{\lambda \sqrt{m}}\right) 
\ge \frac{d(v)\rho \mu^{10}m}{5}.$$ 
Each such switch yields a graph in ${\mathcal F}_{i+1}$.
There are at most $(2i+2)m$ switches  from a graph in ${\mathcal F}_{i+1}$ using an edge from $v$ to $S$. So, we obtain 
$|{\mathcal F}_i| \le \frac{|{\mathcal F}_{i+1}|}{10}$.

Thus, whp, for sufficiently large $D$, $\p{{\mathcal F}_0}=o(\frac{1}{m})$ and hence whp every vertex of $S_{\ge D \log m}$ is within distance three  of a vertex of $S$, proving statement (I).

To prove (II) holds, we can restrict our attention to the 
case that $d(S-S_{\ge \frac{m}{ D \log m}})  \le \frac{\rho \mu^{10} m}{2}$.
Applying Proposition \ref{dn-3large}, we see that (II) holds whp if $d_{n-3} \ge \frac{m}{D \log m}$ so we can assume this 
is not the case. This implies   $4n>d(S_{\ge \frac{m}{ D \log m}})  \ge \frac{\rho \mu^{10} m}{2}$  and $d_n \ge \frac{\rho \mu^{10} m}{8}$.   

Moreover, for any vertex 
$v \in S_{\ge n^{1/3}h(n)}-S$, more than  half the neighbours of $v$ (indeed all but 4) are nonadjacent to more than half the neighbours of $n$,
so switching  between graphs for  
which $v$ neighbours  $n$ and those for which it does not, 
we obtain  $\p{vn \not \in E}  \le  \frac{64}{\rho \mu^{10} d(v)} =o\left(\frac{1}{n^{1/3}\sqrt{h(n)}}\right)$. 

So letting $a=|S_{\geq n^{1/3}h(n)}-S|$ it follows that (II)  holds whp  if $a<n^{1/3}\sqrt{h(n)}$. Else, whp 
at least $\frac{a}{2}$ vertices of $S_{\ge n^{1/3}h(n)}-S$ 
are adjacent to $n$.  

For any $v$ in $S_{\ge n^{1/3}h(n)}-S$ and graph $H$ 
such that 
(i) at least $\frac{a}{2}$ vertices of $S_{\geq n^{1/3}h(n)}-S$ 
are adjacent to  $n$, 
and (ii) $\mathrm{dist}(v,n) > 4$,
the number of switchings from $H$  to graphs for which 
(i) holds and for which $vn \not\in E$  but $v$ has a unique neighbour in $S_{\ge n^{1/3}h(n)} \cap N(n)$ is at least $\left(\frac{a}{2}(n^{1/3}h(n)-1)\right)(n^{1/3}h(n)) = \Omega(an^{2/3}h(n)^2)$. For any graph 
where $v$ has only one common neighbour with  $n$, there 
are at most $2m=O(n)$ switches using the edge from $v$ to this common neighbour.
Since $\p{vn \not \in E} = o\left(\frac{1}{n^{1/3}\sqrt{h(n)}}\right)$, it follows that the probability that (i) holds and there is a vertex of  $S_{\geq n^{1/3}h(n)}-S$  at distance at least 4 from $S$  is $o(1)$. Since (i) holds whp, this implies (II) holds whp as well.  
\end{proof}

\subsection{Proof of Theorem~\ref{thm:mainnew1}}
We now have everything we need to prove our main theorem. For the reader's convenience, we recall the precise statement here.
\begin{customthm}{\bf \ref{thm:mainnew1}}
For any $h$ going to infinity with $n$, the following holds. Let $(\cD_{\ell})_{\ell \ge 1}$ be a sequence of feasible degree sequences such that  $n_0(\cD_{\ell})=0$ and $G(\cD_{\ell})$  has a giant component. Then whp  
    the diameter of every  component  is $O(\frac{m \log m^{\neq 2}}{m^{\neq 2}})$  and all  but the largest contain at most one cycle and  have $O(\max\{n^{1/3}h(n), \frac{m \log m^{\neq 2}}{m^{\neq 2}}\})$ vertices. 
\end{customthm}

\begin{proof}
From Theorem~\ref{whenhasgiantthm} we know that $R_{\cD} \ge \rho m^{\ne 2}$.
Recall that $S=S_{\ge\lambda \sqrt{m}}$ is the set of vertices of degree at least $\lambda \sqrt{m}$, 
where $\lambda>0$ was chosen sufficiently small in terms of $\mu$ and $\rho$ (and $\mu$ was chosen sufficiently small in terms of $\rho$). 

We  partition our sequence of degree sequences into three subsequences,
one consisting of those without a $\mu$-center, 
one  consisting of those with a $\mu$-center in which $d(S) \ge \rho \mu^{10} m$ and one consisting of those with a $\mu$-center in which  $d(S) < \rho\mu^{10}  m$.
We prove Theorem \ref{thm:mainnew1} for all three  subsequences which proves it for the whole sequence. 
For degree sequences without  a $\mu$-center, we simply apply  \Cref{thm:mainnew5}.

{\bf Case 1:}  The subsequence has a $\mu$-center and $d(S) <  \rho\mu^{10} m$.

We expose all the edges leaving $S$ and delete the set $Z$ of vertices of degree zero in $G-S$; this determines the degree sequence ${\cal D}'$ of $(G-S)-Z$ on vertex set $V'$. Let $d'_i$ be the degree of $i$ in this degree sequence. We note that conditioned on the set of edges leaving $S$, $Z$ is the set of vertices of $G-S$ all of whose neighbours lie in $S$ and $G[V']$ is a uniformly random simple graph on ${\cal D'}$. Furthermore, if $H$ is the component of $G$ containing $S$, then every component of $G[V']$ is either a component of $G$ or a component 
of $H-S$, so an understanding of $G[V']$ will allow us to complete the proof.

We claim  that $\D'$ satisfies the conditions for the existence of a giant component.
Each entry in $\cD'$ which is not equal to 2 corresponds either 
 to a vertex of degree not equal to  2  in $\cD$ or to a vertex of degree 1 in $G - S - Z$; there are at most $d(S)$ vertices of the latter type. Since every vertex of $S$ 
 had degree not equal to  2 in $\cD$ but does not appear in $\cD'$, we have $m^{\neq 2}-4d(S) \le  m^{\neq 2}({\mathcal D}') \le m^{\neq 2}$.

Recall that $R_{\cD} = \sum_{i > j_{\cD}} d_i > \rho m^{\neq 2}$. Thus to prove our claim, we need 
only show $R_{\mathcal D'}>\frac{\rho m^{\neq 2}({\mathcal D}')}{4}$.
To see this we consider 
separately the case where $d_{j_{\mathcal D}} \ge \rho^{-5}$ and the case where $d_{j_{\mathcal D}} \le \rho^{-5}$.

In the first case we consider the set $U$ of vertices of degree 
at least $\frac{\rho^{-5}}{2}$ in $\cD'$, which form a suffix in our ordering of $\cD'$. Since
$$
\sum_{i > j_{\D} : d_i' > d_i/2} d_i' \geq 
\sum_{i > j_{\D}} d_i 
-4 d(S)
\geq
\frac{\rho m^{\neq 2}}{2}
$$
which sums over a subset of $U$, we must also have $\sum_{i \in U} d_i' \geq \frac{\rho m^{\neq 2}}{2}$.
On the other hand, since $d'_i(d'_i-2)$ is nonnegative unless $d'_i=1$, in which case $d_i'(d_i'-2)=-1$, and since there are at most $2m^{\neq 2} (\cD') \le 2m^{\neq 2}$ vertices of degree 1, we also have
$$\sum_{i \in V'-U} d_i'(d_i'-2) \ge -2m^{\neq 2}\ .$$

Let $i^*$ be the largest $i \in U$ such that $\sum_{i \in U, i > i^*}d_i' \geq \frac{\rho m^{\neq 2}}{4}$. Then $\sum_{i \in V' - U} d_i'(d_i'-2) + \sum_{i \in U, i < i^*}d_i'(d_i'-2) \geq -2m^{\neq 2} + \left(\frac{\rho^{-5}}{2} - 2\right)\sum_{i \in U, i < i^*} d_i' > 0$ since we assume $\rho$ is sufficiently small, so $j_{\cD'} \leq i^*$. Thus, $R_{\cD'} \geq \sum_{i \in U, i > i^*} \geq \frac{\rho m^{\neq 2}}{4}$ as desired.

For the second case, if $d_{j_{\mathcal D}} \le \rho^{-5}$, 
we consider how the sum $R_{\cD}$ changes upon the deletion of $S$. 
We note that since $S$ is the set of highest-degree vertices, our lower bound on $R_\cD$ and our upper bound on the sum of the degree of the vertices of $S$
ensures that for all $i \in S$, we have $i >j_{\cD}$.  By our assumptions for this case, each time we decrease the degree of some $i \le j_\cD$ by one, $d_i(d_i-2)$ decreases by less than $4\rho^{-5}$.
So, for $i \leq j_{\D}$, we have $\deg_{\D'}(i)(\deg_{\D'}(i)-2)$ is at least $d_i(d_i-2) - 4\rho^{-5}\deg_{\D}(i, S)$ where $\deg_{\D}(i, S)$ is the number of edges from $i$ to $S$.

By definition of $j_{\D}$, we know $\sum_{i \leq j_{\D}} d_i(d_i-2) \geq 0$ so
\begin{align*}\sum_{i \le j_\cD} \deg_{\D'}(i)(\deg_{\D'}(i)-2) &\ge \sum_{i \leq j_{\D}} (d_i(d_i-2) - 4\rho^{-5}\deg_{\D}(i,S))\\
&\geq 0- 4\rho^{-5}\sum_{i \in S} d_i\\
&\ge 4\rho^{-4}\mu^{10}m \ge -\rho \mu^9 m\ .\\
\end{align*}

Now let $T=S\cup \{i > j_{\cD}: \deg_{\D'}(i) < d_{j_{\cD}}\}$, and let $U=\{j_\cD+1,\ldots,n\}-T$. Then $|T \setminus S| \le d(S)$ and for $i \in T \setminus S$, we have $d_i \leq \deg_{\D'}(i) + \deg_{\D}(i,S) \leq d_{j_{\D}} + \deg_{\D}(i,S)$. Thus,
\begin{align*} 
\sum_{i \in T} \deg_{\D'}(i) &\leq \sum_{i \in  T} d_i =\sum_{i \in T \setminus S} d_i +d(S) \leq (|T \setminus S|d_{j_\cD}+d(S) )+d(S)\\
&\le 2d(S)\rho^{-5} \le \rho\mu^9 m\ .
\end{align*} 
This implies that since $\deg_{\D'}(i)(\deg_{\D'}(i)-2) \ge -1 \ge -\deg_{\D'}(i)$, we have $\sum_{i \in T} \deg_{\D'}(i)(\deg_{\D'}(i) - 2) \geq -\rho\mu^9 m$.
Since the degree sequence $\cD$ has a center, we also have $d(U)=R_\cD-d(T)- d_{j_\cD} > \frac{99R_\cD}{100}$.  Now, by definition $U$ forms a suffix of $\cD'$,
and each vertex $i$ of $U$ satisfies $\deg_{\D'}(i)-2 \ge 1$.
If $d_{j_{\D'}} \leq \deg_{\D'}(u)$ for all $u \in U$, then we are done. Otherwise, let $U_< \subset U$ be the vertices of $U$ that occur before $j_{\D'}$ in $\D'$. Then $R_{\D'} \geq \deg_{\D'}(U) - \deg_{\D'}(U_<)$. Since $\sum_{i<j_{D'}} d'_i(d'_i-2) \leq 0$, we have 
\begin{align*} \deg_{\D'}(U_<) &\leq \sum_{i \in U_<} \deg_{\D'}(i)(\deg_{\D'}(i)-2) \\
&\leq -\sum_{i \leq j_{\D}} \deg_{\D'}(i)(\deg_{\D'}(i)-2) - \sum_{i \in T} \deg_{\D'}(i)(\deg_{\D'}(i)-2) \\
&\leq 2\rho \mu^9 m\ .
\end{align*}

It follows that $R_{\cD'} \ge d(U)-(2\rho\mu^9 m) \ge  \frac{98R_\cD}{100}$, so again 
$R_{\cD'}>\rho m^{\neq 2}/4 $.  
This proves our claim.

In deleting $S$, we have changed both $m^{\neq 2}$ and $m$ by at most $4 \rho \mu^{10}m$
 so we  have
$m^{\neq 2}({\mathcal D'} ) \ge \frac{\mu^3 m({\mathcal D}')}{2}$.  
If we choose $\lambda$ small enough in  terms of $\mu, \rho/4$, and $\mu_{\rho/4}$ we can  apply Theorem 
\ref{thm:mainnew5}, to obtain that  its components have diameter $O(\log m)$ and all but the largest have size $O(\log m)$. 

By Lemma \ref{lastlem}, any  two vertices  in the same component are  joined by a path $P$  such that the components  of $P-S$ containing an endpoint of $P$ 
are shortest paths  of $G-S$ between their endpoints and the rest of $P$ 
has at most $g(m)=O(\log \log m)$ edges . This implies $P$ has length $O(\log m)$.  It remains to show 
that there is a unique  component of size exceeding $n^{1/3}h(n)$
which is the  one containing $S$. 

We note that by Theorem \ref{whenhasgiantthm}, the largest component 
of $G-S$ contains $\epsilon' m$ edges for some $\epsilon'>0$ and so we need only show that there is no such component in $G-S$ that has no edges 
to $S$. In order to do so, we show that the expected number of edges 
in such a component is $o(m)$ from which it follows that the 
probability there is such a component is $o(1)$.
Thus, we need only show that each vertex  of $G-S$ is in such a component 
with probability $o(1)$. 

Let $\cF_v$ be the collection of graphs $H$ for which $H-S$ contains some component $K$ that has $\epsilon' m$ edges, has no edges to $S$ in $H$, and contains the vertex $v$. For every $H \in \cF_v$, there are  at least $\epsilon' \lambda m^{3/2}$ switchings using an edge $xy$ of $K$ with $x$ no further from $v$ than $y$ and an edge $us$ with $s \in S$. This yields a graph not in $\cF_v$ where the first edge of every shortest path from $S$ to $v$  is either $xs$ or $uv$. For each such graph there are at most $2m$ switchings into $\cF_v$ using the first edge of a shortest path from $S$ to $v$. Thus $\p{v \in K} \leq \frac{2}{\epsilon' \lambda \sqrt{m}} = o(1)$ so we are done.

{\bf Case 2: }  The subsequence has a $\mu$-center and  $d(S) \ge  \rho \mu^{10} m$. 

Note that since the degree sequence has a $\mu$-center, we must have $m^{\neq 2}>\frac{\rho \mu^{10} m}{2}$, and in particular $m^{\neq 2} = \Omega(m)$. 

We set $d^*$ to be $D \log m$ if $d(S-S_{\ge  \frac{m}{ D \log m}}) \ge \frac{\rho \mu^{10} m}{2}$. Otherwise, we know $m \le n \log n$, and  we set $d^*=n^{1/3}h(n) \le m^{1/3}h(n)<m^{3/8}$. By \Cref{notlastlem}, we note that whp every vertex of  $S_{\ge d^*}$ is within distance 
6 of $S$. 

Our next step is to show that whp every vertex $v$ of $V-S_{\ge d^*}$ 
is either within distance $O(\log m)$ of a vertex of $S_{\ge d^*}$ or in a component  containing $O(\log m)$ vertices of degree at least $2$ and hence of diameter $O(\log m)$. Assuming this holds, then since whp every two vertices of $S_{\geq d^*}$ are within distance 6 of $S$ and any two vertices of $S$ are within distance $g(m) = \log \log m$ of one another, it follows that whp there is a component containing all the vertices of $S_{\ge d^*}$ which has diameter 
$O(\log m)$. This proves the diameter bound in Theorem \ref{thm:mainnew1}.

Now, and again later in the proof, for each vertex $v$ in $V-S_{\ge d^*}$, we  carry out an exploration  of  the  type discussed earlier, by building a tree $T$ rooted at $v$, 
a subgraph $H$ of $G$, and the submatching $M_H$  of $\cM_\cD$ corresponding to $H$ in the following way. 
If we add a vertex $w$ with  $2 \le d(w) \le d^*$ to $T$ in the current iteration, then in the next iteration, we find the unmatched halfedge incident to $w$ with the lowest index and expose the halfedge to which it is matched.
If we do not add such a vertex in the current iteration, then in the next iteration we find the lowest-indexed vertex in $T$ that is incident to at least one unmatched halfedge. For the lowest-indexed unmatched halfedge incident to this vertex, we expose the halfedge to which it is matched.

We carry out our exploration
from $v$ until one of the following holds:  \begin{enumerate}[(a)] 
\item we have exposed the whole component containing $v$, 
\item we have added  $D \log m$ vertices of degree at least two to $T$, or  
\item we have added a vertex of 
degree at least $d^*$ to $T$. 
\end{enumerate}

Since each vertex of $V-S_{\ge d^*}$ is adjacent to fewer than half the 
neighbours of any vertex of $S$ (recalling that $S$ consists of vertices with degree at least $\lambda \sqrt{m}$), switching with edges from 
$S$ shows that if vertex $w$ with $2 \leq d(w) \leq d^*$ is added to the tree in the current iteration, the probability that (conditioned on the exploration up to this point) $w$ matches to a vertex
of $S_{\ge d^*}$ in the next iteration is at least $\frac{\rho\mu^{10}}{4}$. Hence for large $D$, the probability that 
(b) occurs is at most $(1-\frac{\rho \mu^{10}}{4})^{D \log m} = o(1/m)$, so whp there is no $v$ for which (b) occurs. 

If (c)  occurs, then $v$ is within distance $D \log m+1$ of a vertex of 
$S_{\ge d^*}$. This shows that whp every vertex  of $V-S_{\ge d^*}$ is either within distance $O(\log m)$ of a vertex of $S_{\ge d^*}$ or in a component  containing $O(\log m)$ vertices of degree at most $2$, as desired. 

It follows that whp every component except the one containing $S$ fails to intersect $S_{\geq d^*}$  and has at most $O(n^{1/3}h(n)\log m)$ edges and hence $o(n)$ vertices and $o(m)$ 
edges. Since whp $G$ contains a component with $\epsilon n$ vertices for some $\epsilon > 0$, then whp the component containing $S$ is the largest component. It remains to show that the other components have size
$O(n^{1/3}h(n))$ and are unicyclic.

Regarding the size bound, since whp each of these components  has $O(\log m)$  vertices of degree at most two,  whp they each 
contain $O(d^*\log m)$ vertices overall. So we need only consider  $d^*=n^{1/3}h(n)$, i.e. the case that $d(S-S_{\ge \frac{m}{D \log m}})<\frac{\rho\mu^{10} m}{2}$. Furthermore, by \Cref{dn-3large}, whp if $d_{n-3}>\frac{m}{D \log m}$ then none of these components contain vertices of degree exceeding $\frac{n^{1/3}}{{\log m}}$ so whp they all contain $O(n^{1/3})$ vertices and we are done.  Thus, in proving the size bound, we can assume  $d_{n-3}<\frac{m}{D \log m}$ and $d(S-S_{\ge\frac{m}{D \log m}})<\frac{\rho\mu^{10} m}{2}$ which imply $m=O(n)$.

We now   consider an  exploration from $v$ in $V-S_{\ge d^*}$ 
in a breadth-first search manner as before: in each iteration, of the vertices which still have unmatched halfedges, we take the vertex from the earliest iteration, and of the  unmatched halfedges at this vertex we take the one with the lowest index and expose its partner. We continue until  
\begin{enumerate}
    \item we have exposed the whole component containing  $v$, 
    \item we have added four vertices of degree exceeding $\frac{n^{1/3}}{\log m}$, 
    \item we add an edge to a vertex of $S_{\ge d^*}$, 
    \item there are more than $D \log m$ vertices of degree at least 2  in the tree, or 
    \item the explored graph $H$ first contains more than one cycle. 
\end{enumerate}
We have already shown that whp neither  (3) or (4)  will occur for any $v \in V-S_{\ge d^*}$  which is in a component disjoint from $S$, so we need only show that whp  (2) does not occur for any $v$ and that every  $v$  for which  (5) occurs is in a component containing a vertex of $S$. 

If $d_{n-3}<\frac{m}{D \log m}$ and $d(S-S_{>\frac{m}{D \log m}})<\frac{\rho\mu^{10} m}{2}$ then as we have already seen,  $m=O(n)$  and $d_n=\Omega(m)$. Because we terminate if (4) occurs, the probability that a vertex $w$ which is added to the tree is not adjacent to $n$ (conditioned on the exploration up to the point at which we add $w$, which includes the neighbours of $w$  added to the tree before $w$) is $O(\frac{1}{d(v)-O(\log m)})$. It follows that the probability (2) occurs is $O((\log m)^4(\frac{\log m}{n^{1/3}})^4)=o(\frac{1}{n})$. Thus whp there is no $v$ for which (2) occurs. 

If (5) occurs but no vertex of the constructed tree $T$ has an edge 
to $S$, let $w$ be the last vertex we explored from and $wy$ be the last edge we explored, so in particular $y \in T$. Also consider the first cycle contained in $H$: let $ab$ be the last edge that completed this cycle, meaning $a, b \in T$ but $ab \notin T$, 
labelled so that $a$ was added to $T$ before $b$.

Note that we have not yet explored from $y$, as we explore all the edges from a vertex consecutively except the one causing it to be added to $T$. This implies that $y$ was added to $T$ after $w$ and is not $a$. We also note that since the exploration of $wy$ causes $H$ to first contain more than one cycle, $y$ was in $T$ before we started exploring from $w$ and has degree at least two. Since $y$ is not adjacent to any vertex of $S$ by assumption and since $\deg(w) \leq d^*$, for every vertex $s$
of $S$, we can switch $sz$ with $wy$ for at least half the neighbours $z$ of 
$s$.  There are $\Omega(m)$ such switches.  

Now consider what happens with the exploration process in the graph $G'$ which results from switching $sz$ with $wy$; let $T'$ be the corresponding tree in $G'$ constructed when exploring from $v$. Observe that such a switch does not change the graph constructed up to the iteration at which we revealed the halfedge at $w$ pairing with a halfedge at $y$.  Furthermore, the tree $T'$ up to revealing $wy$ coincides with $T$.  
So in $G'$, either we stop exploring because we reveal $wz$ for $z \in S$ (causing (3) to occur) and $ys$ is the only other edge in $G'$ from $T'$ to $S$,
or $y$ is the first vertex in $G'$ we explore which has an edge to $S$ and it has only one such edge. This means that in $G'$, there at most 2 choices for such an edge $ys$. 
To switch in the other direction we must use one of these two choices and one of the $O(d^* \log m)$  edges incident to the first $D \log m$ 
vertices  of degree at least two added to the tree in our exploration from $v$. 

After carrying out the switch from $G$ to $G'$, the edge $ab$ is (i) still an edge between two of the $\leq$ $D \log m$ vertices of degree at least 2 in the tree $T'$, and (ii) not in the tree $T'$. 

If $a \neq w$ or  $z \not \in S$, then $a$ is still adjacent to none of $S$. In this case  we consider  switchings  using  $ab$  from this new graph $G'$ to a newer graph $G''$ in which $a$ is adjacent to a unique vertex of $S$. There are $\Omega(m)$ such switches. Now, $a$ is the last vertex of the tree $T''$ that we explore from and has a unique edge to $S$, and $b$ is one of the first $D \log m$ vertices of degree at least two we explore. 

Else, if $a = w$ and $z \in S$, then $b$ is neither $w$ nor $y$ and so is adjacent to none of $S$. In this case, we consider  switchings  on $ab$ from this new graph $G'$ to a newer graph $G''$ in which $b$ is adjacent to a  unique vertex of $S$. As before, there are $\Omega(m)$ such switches. Again $a$ is the last vertex of the tree $T''$ that we 
explore from.  Now,  $b$ is one of the 
$D \log m$ vertices of degree at least two in the tree   we create. 
Furthermore, there are at most four edges from this tree to $S$. 

So, in either case, there  are  $O(d^* \log m)$ choices for switching in the other direction from $G'$ to a graph in which no vertex of the explored tree has an edge to $S$. Since $(\frac{d^* \log m}{m})^2 = o(m^{-1})$, whp
there is no $v$ for which (5) occurs and $v$ is not in the component containing $S$. 
\end{proof}

\section{Acknowledgments}
This work began while the authors were in residence at the Simons--Laufer Mathematical Sciences Institute during the Spring 2025 semester, supported by NSF grant DMS-1928930. 
LAB is supported by an NSERC Discovery Grant and by the Canada Research Chairs program. 
BR is supported by NSTC Grant 112-2115-M-001 -013 -MY3.

\bibliographystyle{abbrv}
\bibliography{references}
                
\appendix

\end{document}